\documentclass[a4paper,12pt]{article}
\usepackage{amsmath}
\usepackage{amssymb}
\usepackage{amsthm}
\usepackage{amsfonts}
\usepackage{mathscinet}
\usepackage{enumerate}
\usepackage{pdfsync}
\usepackage{setspace}
\usepackage{mathrsfs}
\usepackage{stmaryrd}
\usepackage{tikz}
\usepackage{cite}

\usepackage{bookmark}
\usepackage{authblk}

\newcommand {\new} {\newcommand}
\newcommand \oper[2] {\new #1 {\operatorname{#2}}}

\newcommand \gcode[1] {\ulcorner\! #1 \!\urcorner}	
\newcommand \se [1] { \{ #1 \}}			
\newcommand\set [2]{ \{#1:#2\} }			
\newcommand\res {\!\upharpoonright\!}		

\newcommand\elem {\prec}				
\newcommand\eqiv {\leftrightarrow}			
\newcommand\corner[1] {
  \langle #1 \rangle}		

\newcommand\power {{\mathscr{P}}}			

\newcommand\coll{{\text{Coll}}}

\newcommand{\card}{\operatorname{card}}

\newcommand\concat {{^\frown}}   

\oper{\Ord}{Ord}				
\oper{\hull}{Hull}				
\oper{\ppt}{ppt} 				
\newcommand{\ord} {\Ord}
\oper{\ZFC}{ZFC}				
\oper{\rank}{rank}				
\oper{\crit}{cr}					
\oper{\crt}{crt}					
\oper{\cf}{cf}					

\oper{\height}{ht}				
\oper{\wfcore}{wfcore}					
\oper{\core}{core}					

\oper{\Ult}{Ult}				
\oper{\Cone}{Cone}				
\oper{\dirlim}{dirlim}
\oper{\rud}{rud}		
\oper{\const}{const}
\oper{\OD}{OD}
\oper{\final}{final}
\oper{\HYP}{HYP}
\oper{\wfp}{wfp}
\oper{\Hom}{Hom}
\new{\ult} {\Ult}
\oper{\dom}{dom}
\oper{\rep}{rep}

\newcommand{\bardom}{\dom^{*}}
\newcommand{\exdesc}{\desc^{*}}

\oper{\suc}{succ}
\oper{\fac}{fac}

\oper{\Code}{Code}

\oper{\ran}{ran}
\oper{\maxdom}{maxdom}
\oper{\maxran}{maxran}

\oper{\lp}{Lp} 

\newcommand{\Los}{{\L}o\'{s}}
\newcommand\lengtheneq {\vartriangleright}	
\newcommand\lengthen {\triangleright}		
\newcommand\inisegeq {\trianglelefteq}	

 \newtheorem{mytheorem}{Theorem}[section]
 \newtheorem{myproposition}[mytheorem]{Proposition}
 \newtheorem{mylemma}[mytheorem]{Lemma}

 \newtheorem{mycorollary}[mytheorem]{Corollary}

 \newtheorem{myclaim}[mytheorem]{Claim}
 \theoremstyle{mydefinition}\newtheorem{remark}[mytheorem]{Remark}
 \newtheorem{mydefinition}[mytheorem]{Definition}

\newcommand{\game}{{\Game}}
\newcommand{\boldpi}[1]{{\boldsymbol{\Pi}^1_{#1}}}
\newcommand{\boldsigma}[1]{{\boldsymbol{\Sigma}^1_{#1}}}

\newcommand{\boldgameclass}[1]{{\game ^{#1}(<\!\omega^2\textnormal{-}\mathbf\Pi{^1_{1}})}}

\newcommand{\bolddelta}[1]{{\boldsymbol{{\delta}}^1_{#1}}}
\newcommand{\boldDelta}[1]{{\boldsymbol{{\Delta}}^1_{#1}}}
\newcommand{\admistwo}[1]{L_{\kappa_3^{#1}}[T_2,#1]}

\newcommand{\admistwobold}{{  \mathbb{L}_{\boldsymbol{{\delta}}^1_3}[T_2] }}

\newcommand{\WO}{{\textrm{WO}}}
\newcommand{\LO}{{\textrm{LO}}}
\newcommand{\DEF}{=_{{\textrm{DEF}}}}

\newcommand{\desc}{{\operatorname{desc}}}

\newcommand{\Diff}{\operatorname{Diff}}

\newcommand{\id}{\operatorname{id}}
\newcommand{\tree}{\operatorname{tree}}
\newcommand{\node}{\operatorname{node}}
\newcommand{\seed}{\operatorname{seed}}
\newcommand{\pred}{\operatorname{pred}}

\newcommand{\ot}{\mbox{o.t.}}

\newcommand{\sharpcode}[1]{  \left|   #1  \right|}

\newcommand{\absvalue}[1]{  \left|   #1  \right|}

\newcommand{\wocode}[1]{  \|   #1  \|}

\newcommand{\lessthanshort}[1]{{<\! #1}}

\newcommand{\comm}[1]{{}}
\newcommand{\comp}[2]{{}^{#1}\!#2}
\oper{\ucf}{ucf}					
\oper{\sign}{sign}					
\oper{\lh}{lh}

\begin{document}

\title{The higher sharp II: on $M_2^{\#}$}

\author{Yizheng Zhu}

\affil{Institut f\"{u}r mathematische Logik und Grundlagenforschung \\
Fachbereich Mathematik und Informatik\\
Universit\"{a}t M\"{u}nster\\
Einsteinstr. 62 \\
48149 M\"{u}nster, Germany}

\maketitle

\begin{abstract}
We establish the descriptive set theoretic representation of the mouse $M_n^{\#}$, which is called $0^{(n+1)\#}$. This part partially deals with the case $n=2$ by proving the many-one equivalence of $M_2^{\#}$ and the theory of $L_{\bolddelta{3}}[T_3]$ with the higher level analogs of $L$-indiscernibles.
\end{abstract}

\section{Introduction}
\label{sec:introduction}

This is the second part of a series starting with \cite{sharpI}. It defines the higher level analogs of order indiscernibles for $L$. They are level-3 indiscernibles for $L_{\bolddelta{3}}[T_3]$. The theory of $L_{\bolddelta{3}}[T_3]$ with these level-3 indiscernibles will be called $0^{3\#}$. We will then show its many-one equivalence with $M_2^{\#}$.

As advertised in the introduction of \cite{sharpI}, the structure of the level-3 indiscernibles for $L_{\bolddelta{3}}[T_3]$ resembles structure of the $\mathbb{L}[T_3]$-homogeneous trees on $\omega \times \bolddelta{3}$ that project to a $\Pi^1_3$ set. Under AD, these trees are defined in \cite{kechris_homo_tree_cabal} and \cite{cabal2_intro_jackson}. In this paper, we will define them again using slightly different notations. 
To a new reader, the combinatorial definitions with homogeneous trees in this paper may seem like  unnecessarily complications of very simple facts. However these definitions will fit in well with the generalized Jackson's analysis in the third paper of this series. We urge the reader to bear with the cumbersome notations and possibly have a couple of simple examples in mind. The price is very low, whereas the effort will pay off in the third paper of this series. 

\section{Backgrounds and preliminaries}
\label{sec:KM}

All the notations of this paper will follow \cite{sharpI}. We introduce additional background knowledge for this paper. 

Suppose $A \subseteq \mathbb{R}$. A norm on $A$ is a function $\varphi : A \to \ord$. $\varphi$ is regular iff $\ran(\varphi)$ is an ordinal. A scale on $A$ is a sequence of norms $\vec{\varphi} = (\varphi_n)_{n<\omega}$ on $A$ such that if $(x_i)_{i < \omega} \subseteq A$, $x_i \to x (i \to \infty)$ in the Baire topology, and for all $n$, $\varphi_n(x_i) \to \lambda_n (i \to \infty)$ in the discrete topology, then $x \in A$ and $\forall n~\varphi_n (x)\leq \lambda_n$. 
$\vec{\varphi}$ is regular iff each $\varphi_n$ is regular. If $A = p[T]$, $T$ is a tree on $\omega \times \lambda$, the $\lambda$-scale associated to $T$ is $(\varphi_n)_{n<\omega}$ where $\varphi_n ( x ) = \corner{\alpha_x^0,\dots,\alpha_x^n}$, $(\alpha_x^n)_{n < \omega}$ is the leftmost branch of $T_x \DEF \set{\vec{\beta}}{ (x, \vec{\beta}) \in [T]}$, $\corner{\dots} : \lambda^{n+1} \to \ord$ is order preserving with respect to the lexicographic order and is onto an ordinal. 
Suppose $\Gamma$ is a pointclass. If $\varphi$ is a norm on $A$, then $\varphi$ is a $\Gamma$-norm iff the relations
\begin{align*}
 x \leq _{\varphi} y  & \eqiv   x \in A \wedge (y \in A \to \varphi(x) \leq \varphi(y)),
\\
 x < _{\varphi} y  & \eqiv   x \in A \wedge (y \in A \to \varphi(x) < \varphi(y)).
\end{align*}
are both in $\Gamma$. $\vec{\varphi} = (\varphi_n)_{n<\omega}$ is a $\Gamma$-scale iff the relations $x \leq_{\varphi_n} y$ and $x <_{\varphi_n } y$ in $(x,y,n)$ are both in $\Gamma$. $\Gamma$ has the prewellordering property iff every set in $\Gamma$ has a $\Gamma$-norm. 
$\Gamma$ has the scale property iff every set in $\Gamma$ has a $\Gamma$-scale. Assuming PD, Moschovakis \cite{mos_dst} shows that the pointclasses $\Pi^1_{2n+1}$, $\boldpi{2n+1}$, $\Sigma^1_{2n+2}$, $\boldsigma{2n+2}$ have the scale property.

For  a nonempty finite tuple $t = (a_0,\ldots,a_k)$, put $t^{-} = (a_0,\ldots,a_{k-1})$. 
This notation will be followed throughout this paper. 
If $<_i$ is a linear ordering on $A_i$ for $i<\omega$, then $<_{BK}^{(<_i)_i} $ is the Brouwer-Kleene order on $\bigcup_{n<\omega} (\Pi_{i < n} A_i)$ where $(a_0,\dots, a_n) <_{BK}^{(<_i)_i} (b_0,\dots,b_m)$ iff either $(a_0,\dots,a_n)$ is a proper lengthening of $(b_0,\dots,b_m)$ or there exists $k\leq\min(m,n)$ such that $\forall i < k ~a_i = b_i \wedge a_k <_k b_k$. In our applications, these orderings $<_i$ will be apparent enough so that $(<_i)_i$ can be omitted from the superscript without confusion. 

Put $\mathbb{L} = \bigcup_{x \in \mathbb{R}}L[x]$, $\mathbb{L}_{\alpha} = \bigcup_{x \in \mathbb{R}}L_{\alpha}[x]$. If $A$ is a set, put $\mathbb{L}[A] = \bigcup_{x \in \mathbb{R}}L[A,x]$, $\mathbb{L}_{\alpha}[A] = \bigcup_{x \in \mathbb{R}}L_{\alpha}[A,x]$.
$\mathbb{L}$ and $\mathbb{L}[A]$ are in general not models of $ZF$. 
Nonetheless, cardinality and cofinality in $\mathbb{L}[A]$ are well defined. So for example,  $\cf^{\mathbb{L}[A]} (\alpha) = \min \set{\cf^{L[A,x]}(\alpha)}{ x \in \mathbb{R}}$. 

If $R$ is a wellfounded relation, $\wocode{x}_R$ denotes the $R$-rank of $x$, i.e., $\wocode{x}_R = \sup \set{\wocode{y}_R+1}{yRx}$. If $<$ is a linear order, then $\pred_{<}(a)$, $\suc_{<}(a)$ denote the $<$-predecessor and $<$-successor of $a$ respectively, if exists. 

\subsection{$Q$-theory}
\label{sec:q-theory}

For $x\in \mathbb{R}$, $\admistwo{x}$ is the minimum admissible set containing $(T_2,x)$. 
We recall the following model theoretic representation of $\Pi^1_3$ sets in \cite{Kechris_Martin_I,Kechris_Martin_II,becker_kechris_1984}. 
\begin{mytheorem}[Becker-Kechris-Martin]\label{thm:BK-KM}
  Assume $\boldDelta{2}$-determinacy. Then for each $A \subseteq u_{\omega} \times \mathbb{R}$, the following are equivalent.
  \begin{enumerate}
  \item $A$ is $\Pi^1_3$.
  \item There is a $\Sigma_1$ formula $\varphi$ such that $(\alpha,x) \in A$ iff $\admistwo{x}\models \varphi(T_2,\alpha,x)$.
  \end{enumerate}
\end{mytheorem}


We will need further results on Theorem~\ref{thm:BK-KM}.
The original proof of 2 $\Rightarrow$ 1 in Theorem~\ref{thm:BK-KM} is based on Theorem~\ref{thm:Kechris_Martin_Pi13_subset_of_u_omega_basis} and Corollary~\ref{coro:Kechris_Martin_closure_under_quantification}. 
\begin{mytheorem}[Kechris-Martin,  \cite{Kechris_Martin_I,Kechris_Martin_II}]
  \label{thm:Kechris_Martin_Pi13_subset_of_u_omega_basis}
  Assume $\boldDelta{2}$-determinacy. Let $x \in \mathbb{R}$.  If $A$ is a nonempty $\Pi^1_3(x)$ subset of $u_{\omega}$, then $\exists w \in \Delta^1_3(x)\cap \WO_{\omega} ( \sharpcode{w} \in A)$.
\end{mytheorem}

\begin{mycorollary}[Kechris-Martin,  \cite{Kechris_Martin_I,Kechris_Martin_II}]
  \label{coro:Kechris_Martin_closure_under_quantification}
  Assume $\boldDelta{2}$-determinacy.  Then $\Pi^1_3$ is closed under quantifications over $u_{\omega}$, i.e., if $A \subseteq (u_{\omega})^2 \times \mathbb{R}$ is $\Pi^1_3$, then so are
  \begin{align*}
    B & = \set{(\alpha,x)}{\exists \beta<u_{\omega} ~ (\beta,\alpha,x) \in A}, \\
    C & = \set{(\alpha,x)}{\forall \beta <u_{\omega}~ (\beta,\alpha,x) \in A}.
  \end{align*}
\end{mycorollary}

Suppose $\mathcal{X}$ is a Polish space. For $x \in \mathbb{R}$ and $\alpha<u_{\omega}$, $A \subseteq \mathcal{X}$ is $\Sigma^1_3(x,\alpha)$ iff there is a $\Sigma^1_3(x)$ set $B \subseteq u_{\omega} \times \mathcal{X}$ such that $y \in A$ iff $(\alpha,y) \in B$. Or equivalently, $A$ is $\Sigma^1_3(x,\alpha)$ iff there is a $\Sigma^1_3(x)$ set 
$B \subseteq \mathbb{R} \times \mathcal{X} $ such that $y \in A$ iff  $\exists w \in \WO_{\omega}~(\sharpcode{w} = \alpha \wedge (w,\alpha) \in B)$. $A$ is $\Pi^1_3(x,\alpha)$ iff $\mathcal{X}\setminus A$ is $\Sigma^1_3(x,\alpha)$. $A$ is $\Delta^1_3(x,\alpha)$ iff $A$ is both $\Sigma^1_3(x,\alpha)$ and $\Pi^1_3(x,\alpha)$. $\Sigma^1_3(x,\lessthanshort{\beta})$ means $\Sigma^1_3(x,\alpha)$ for some $\alpha<\beta$. Similarly define $\Pi^1_3(x,\lessthanshort{\beta})$ and $\Delta^1_3(x,\lessthanshort{\beta})$.

In the proof of Theorem~\ref{thm:BK-KM}, the prewellordering property for $\Pi^1_3$ subsets of $\omega \times u_{\omega}$, originally proved by Solovay,  is used. 

\begin{mytheorem}[Solovay, {\cite[{Theorem 3.1}]{kechris_1978_II}}]
  \label{thm:pwo_subsets_of_u_omega_times_R}
Assume $\boldDelta{2}$-determinacy. Suppose $A \subseteq u_{\omega} \times \mathbb{R}$ is $\Pi^1_3(x,\alpha)$, where $x \in \mathbb{R}$, $\alpha<u_{\omega}$. Then there is a $\Pi^1_3(x,\alpha)$ norm $\varphi: A \to \ord$, i.e., the relations
\begin{align*}
  (\beta,y) \leq^{*}_{\varphi} (\gamma,z) & \eqiv (\beta,y) \in A \wedge ((\gamma,z) \in A \to \varphi(\beta,y) \leq \varphi(\gamma,z))\\
  (\beta,y) <^{*}_{\varphi} (\gamma,z) & \eqiv (\beta,y) \in A \wedge ((\gamma,z) \in A \to \varphi(\beta,y) < \varphi(\gamma,z))
\end{align*}
are $\Pi^1_3(x,\alpha)$.
\end{mytheorem}

We uses the above theorems to establish a $\Sigma^1_3$-boundedness theorem with parameters in $ u_{\omega}$.
\begin{mycorollary}[Reduction]\label{coro:Pi13_reduction}
  Assume $\boldDelta{2}$-determinacy. Suppose $A, B \subseteq u_{\omega} \times \mathbb{R}$ are both $\Pi^1_3(x,\alpha)$, where $x \in \mathbb{R}$, $\alpha<u_{\omega}$. Then there exist $\Pi^1_3(x,\alpha)$ sets $A',B'\subseteq u_{\omega} \times \mathbb{R}$ such that $A'\subseteq A$, $B'\subseteq B$,  $A \cup B = A' \cup B'$ and $A' \cap B' = \emptyset$.
\end{mycorollary}

\begin{mycorollary}[Easy uniformization]
  \label{coro:uniformization_Pi13_u_omega}
Assume $\boldDelta{2}$-determinacy. Suppose $A \subseteq (u_{\omega}\times \mathbb{R}) \times u_{\omega} $ is $\Pi^1_3(x,\alpha)$, where $x \in \mathbb{R}$, $\alpha < u_{\omega}$. Then $A$ can be uniformized by a $\Pi^1_3(x,\alpha)$ function, i.e.,  there is a $\Pi^1_3(x,\alpha)$ function $f$ such that $ \dom(f) = \{(\beta,y) : \exists \gamma~  (  (\beta,y),\gamma) \in A\}$ and that $((\beta,y), f(\beta,y)) \in A$ for all $(\beta,y) \in \dom(f)$. 
\end{mycorollary}

 The $\Pi^1_3$ coding system for $\Delta^1_3$ sets (e.g.,  \cite[Theorem 3.3.1]{HKL_1990}) applies to the larger pointclass $\Delta^1_3(\lessthanshort{u_{\omega}})$. The proof is similar.
\begin{mycorollary}[$\Pi^1_3$-codes for $\Delta^1_3(\lessthanshort{u_{\omega}})$]
  \label{coro:Pi13_index_for_Delta13}
  Assume $\boldDelta{2}$-determinacy.  Then there is a $\Pi^1_3$ set $C \subseteq u_{\omega}$ and sets $P,S \subseteq u_{\omega} \times \mathbb{R}$ in $\Pi^1_3, \Sigma^1_3$ respectively such that for any $\alpha \in C$, 
  \begin{displaymath}
    P_{\alpha} = S_{\alpha} \DEF D_{\alpha}
  \end{displaymath}
and
\begin{displaymath}
\set{ D_{\alpha}}  {\alpha \in C}  = \set{A  \subseteq \mathbb{R}}{ A \text{ is } \Delta^1_3(\lessthanshort{u_{\omega}})}.
\end{displaymath}
\end{mycorollary}
\begin{proof}
Let $U\subseteq \omega \times \mathbb{R}^2$ be a good universal $\Pi^1_3$ set. 
  Define 
  \begin{align*}
    ( (n,\alpha),(m,\beta),x) \in A & \eqiv  \forall w \in \WO_{\omega}~( \sharpcode{w} = \alpha \to (n,w,x) \in U) \\
    ( (n,\alpha),(m,\beta),x) \in B & \eqiv  \forall w \in \WO_{\omega}~( \sharpcode{w} = \beta \to (m,w,x) \in U) 
  \end{align*}
Then $A$, $B$ are $\Pi^1_3$ subsets of $(\omega \times u_{\omega})^2$. Reduce them to $A'$, $B'$ according to Corollary~\ref{coro:Pi13_reduction}. Define
\begin{displaymath}
  ((n,\alpha) ,(m,\beta))\in C \eqiv (A')_{(n,\alpha) ,(m,\beta)} \cup  (B')_{(n,\alpha) ,(m,\beta)} = \mathbb{R}
\end{displaymath}
$C$ is a $\Pi^1_3$ subset of $(\omega \times u_{\omega})^2$. Let $P = A'$, $S = (\omega \times u_{\omega})^2 \times \mathbb{R}\setminus B'$. Identifying $(\omega \times u_{\omega})^2$ with $u_{\omega}$ with the G\"{o}del pairing function, $C,P,S$ are as desired.
\end{proof}

Theorem~\ref{thm:BK-KM} provides a model-theoretic view of $Q$-theory \cite{Q_theory} at the level of $Q_3$-degrees. We give an exposition of these results, probably with simple strengthenings thereof. 

The higher level analog of the hyperarithmetic reducibility on reals is $Q_3$ reducibility. $Q_3$-degrees are coarser than $\Delta^1_3$-degrees.
$y \in Q_3(x) $ iff  $y$ is $\Delta^1_3(x)$ in a countable ordinal, i.e.,  there is  $\alpha<\omega_1$ such that  $ \forall w \in \WO (\sharpcode{w}  = \alpha \to y \in \Delta^1_3(x)$. $y$ is $\Delta^1_3(x)$ in an ordinal $<{u_{\omega}}$ iff there is $\alpha < u_{\omega}$ such that $ \forall w \in \WO_{\omega} (\sharpcode{w}  = \alpha \to y \in \Delta^1_3(x))$. $y \leq_{\Delta^1_3} x$ iff $y \in \Delta^1_3(x)$. $y \equiv_{\Delta^1_3} x$ iff $y \leq _{\Delta^1_3} x \leq _{\Delta^1_3} y$. $y \leq_{Q_3}x$ iff $y \in Q_3(x)$. $y \equiv_{Q_3} x$ iff $y \leq_{Q_3} x \leq _{Q_3} y$. 

\begin{myproposition}[\cite{Q_theory,kechris_analytical_I,Kechris_Martin_I,Kechris_Martin_II,steel_projective_wo_1995}]
  \begin{enumerate}
\item Let $x,y \in \mathbb{R}$. Then $y \in  \admistwo{x}$ iff  $y \in M_1^{\#}(x)$ iff $y$ is $\Delta^1_3(x)$ in a countable ordinal iff $y$ is $\Delta^1_3(x)$ in an ordinal $<u_{\omega}$.  
\item The relation $y \in  \admistwo{x}$ is $\Pi^1_3$, where $x,y$ ranges over $\mathbb{R}$.
\item The relation $y \in  \Delta^1_3(x)$ is $\Pi^1_3$, where $x,y$ ranges over $\mathbb{R}$.
\end{enumerate}
\end{myproposition}

$\kappa_3^x$ is the higher level analog of $\omega_1^x$, the least $x$-admissible. 
It is defined in a different way in \cite[Section 14]{Q_theory}. As in \cite{Q_theory,Kechris_Martin_II}, we define
\begin{align*}
  \lambda_3^x & = \sup \set{ \absvalue{W} }{W \text{ is a } \Delta^1_3(x) \text{ prewellordering on } \mathbb{R}} \\
  & = \sup \set{ \xi < \kappa_3^x}{ \xi \text{ is } \Delta_1\text{-definable over } {L_{\kappa_3^x}[T_2,x]} \text{ from } \se{T_2,x}}.
\end{align*}
The equivalence of these two definitions of  $\kappa_3^x$ is proved in \cite{Kechris_Martin_II}:
 \begin{align*}
  \kappa_3^x &= \sup \set{\ot(W)}{W \text{ is a $\Delta^1_3(x,\lessthanshort{u_{\omega}})$ prewellordering on $\mathbb{R}$}}\\
         &= \sup \set{\lambda_3^{x,y}}{ M_1^{\#}(x) \nleq_{\Delta^1_3} (x,y)}.
  \end{align*}
Moreover,
\begin{displaymath}
  \forall \alpha < u_{\omega} ~\exists w \in \WO_{\omega} ~(\sharpcode{w} = \alpha \wedge \lambda_3^{x,w} < \kappa_3^x) .
\end{displaymath}
Note that $\kappa_3^x < \lambda_3^{M_1^{\#}(x)} < \bolddelta{3}$, as proved in \cite[Lemma 14.2]{Q_theory}. 

The Kunen-Martin theorem implies that  $\kappa_3^x$ is  a bound on the rank of any $\Sigma^1_3(x, \lessthanshort{u_{\omega}})$ wellfounded relation. 

\begin{mytheorem}[Kunen-Martin, {\cite[2G.2]{mos_dst}}]
  \label{thm:Kunen-Martin}
Suppose $W$ is a wellfounded relation on $\mathbb{R}$. Suppose $\gamma$ is an ordinal and $T$ is a tree on $(\omega \times \omega )\times \gamma$ such that $W = p[T]$. Let $L_{\kappa}[T]$ be the least admissible set containing $T$ as an element. Then the rank of $W$ is smaller than
\begin{displaymath}
  \sup \set{\xi < \kappa}{ \xi \text{ is } \Delta_1\text{-definable over } L_{\kappa}[T] \text{ from } \se{T}}.
\end{displaymath}
\end{mytheorem}

\begin{mycorollary}
  \label{coro:Kunen_Martin_Sigma13}
Suppose $W$ is a $\Sigma^1_3(x, \lessthanshort{u_{\omega}})$ wellfounded relation on $\mathbb{R}$. Then the rank of $W$ is smaller than $\kappa_3^x$.
\end{mycorollary}


We finally note down the complexity of subsets of ordinals as a consequence of Theorem~\ref{thm:BK-KM}. Assume $\boldDelta{2}$-determinacy. Every subset of $u_{\omega}$ in $\admistwobold$ is $\boldDelta{3}$. Solovay's game shows that every subset of $\omega_1$ in $\admistwobold$ is in $\mathbb{L}$: Let $A \subseteq \omega_1$ be in $\mathbb{L}_{\bolddelta{3}}[T_2]$. Let $B$ and $C$ be $\boldpi{2}$  such that $w \in \WO \wedge \wocode{w} \in A$ iff $ \exists z (w ,z )\in B$ iff $ \neg \exists z (w , z) \in C$. Play the game in which I produces $v$, II produces $w,z$, and II wins iff $v \in \WO \to (w \in \WO \wedge \wocode{w} \geq \wocode{v} \wedge \forall \alpha < \wocode{w} ~\exists (w' ,z') \leq_T z~(\wocode{w'}  = \alpha  \wedge (w' ,z')\in B \cup C))$. I does not win by boundedness. If $\sigma$ is a winning strategy for II, then $A \in L[\sigma]$.

\subsection{Silver's dichotomy  on \texorpdfstring{$\boldpi{3}$}{} equivalence relations}
\label{sec:silvers-dichotomy}

Harrington's proof \cite{harrington_proof_silver}, \cite[Chapter 32]{jech} of Silver's dichotomy  \cite{silver_dichotomy} on $\boldpi{1}$ equivalence relations generalizes to $\boldpi{3}$ in a straightforward fashion. This folklore generalization is stated  in  \cite{hjorth_LT2,hjorth_coarse}  in a slightly weaker form. 

An equivalence relation $E$ on $\mathbb{R}$ is thin iff there is no perfect set $P$ such that $\forall x, y \in P~(x E y \to x=y)$. If $\Gamma$ is a pointclass, for equivalence relations $E,F$ (possibly on different spaces of the form $\mathbb{R}^m \times (u_{\omega})^n$), $E$ is $\Gamma$-reducible to $F$ iff there is a function $\pi$ in $\Gamma$ such that $x E y \eqiv \pi(x)  F  \pi(y)$. 

\begin{mytheorem}[Folklore]
  \label{thm:silver_dichotomy_pi13}
  Assume $\boldDelta{2}$-determinacy. Let $x \in \mathbb{R}$. If $E$ is a thin $\Pi^1_3(x)$ equivalence relation on $\mathbb{R}$, then $E$  is ${\Delta^1_3}(x) $ reducible to a $\Pi^1_3(x)$ equivalence relation on a $\Pi^1_3(x)$ subset of $u_{\omega}$.
\end{mytheorem}
\begin{proof}
For simplicity, let $x=0$.  The generalization of Harrington's proof of Silver's dichotomy shows that for every $y \in \mathbb{R}$, there is a $\Delta^1_3(\lessthanshort{u_{\omega}})$ set $A$ such that $y \in A \subseteq [y]_E$. 

  Let $C,P,S, (D_{\alpha})_{\alpha \in C}$ be the $\Pi^1_3$ coding system for $\Delta^1_3(\lessthanshort{u_{\omega}})$ subsets of $\mathbb{R}$, given by  Corollary~\ref{coro:Pi13_index_for_Delta13}.
Let $\alpha \in  C'$ iff $\alpha \in  C$ and $\forall y \in D_{\alpha} \forall z \in D_{\alpha} (y E z)$. $C'$ is $\Pi^1_3$. 
%
The set
\begin{displaymath}
A = \set{(y,\alpha)}{\alpha \in C' \wedge y \in D_{\alpha}}
\end{displaymath}
is $\Pi^1_3$. By Corollary~\ref{coro:uniformization_Pi13_u_omega}, $A$ can be uniformized by a $\Pi^1_3$ function $\pi$. Let  $\alpha F \beta $ iff $\alpha\in C'$, $\beta\in C'$, and $\forall y \in D_{\alpha} \forall z \in D_{\beta} (y E z)$. $F$ is a $\Pi^1_3$ equivalence relation on $C'$. 
$\pi$ is a reduction from $E$ to $F$. To see that $\pi$ is also $\Sigma^1_3$, apply Corollary~\ref{coro:Kechris_Martin_closure_under_quantification} and use the fact that $\pi$ is a total function taking values in $u_{\omega}$.
\end{proof}
 
The reduction $\pi$ and the target equivalence relation $F$ in Theorem~\ref{thm:silver_dichotomy_pi13} are uniformly definable from the $\Pi^1_3(x)$ definition of $E$, independent of $x$. 
A similar uniformity applies to the following corollary.

\begin{mycorollary}
  \label{coro:Silver_dichotomy_Delta13}
    Assume $\boldDelta{2}$-determinacy. Let $x \in \mathbb{R}$. If $E$ is a thin $\Delta^1_3(x)$ equivalence relation on $\mathbb{R}$, then $E$  is ${\Delta^1_3}(x) $ reducible to  $=_{u_{\omega}}$. Here $\alpha =_{u_{\omega}} \beta$ iff $\alpha = \beta < u_{\omega}$. 
\end{mycorollary}
\begin{proof}
Assume $x=0$.
  Proceed as in the proof of Theorem~\ref{thm:silver_dichotomy_pi13} until we reach the set $A$. We now show that $A$ can be uniformized by a $\Pi^1_3$ function $\pi$ such that $y E z$ iff $\pi(y) = \pi(z)$. Indeed, let $\varphi$ be a $\Pi^1_3$-norm on $A$, given by Theorem~\ref{thm:pwo_subsets_of_u_omega_times_R}, and let $\pi(y)= \alpha$ iff $(y,\alpha) \in A$ and $(\varphi(y,\alpha),\alpha)$ is lexicographically minimal among the set $\set{(\varphi(z,\beta),\beta)}{ z E y  \wedge (z,\beta) \in A}$. Similarly to the proof of Corollary~\ref{coro:uniformization_Pi13_u_omega}, $\pi$ is $\Pi^1_3$ (we use $E \in \Delta^1_3$ here).  Again, $\pi$ is $\Sigma^1_3$. $\pi$ is the desired reduction from $E$ to $=_{u_{\omega}}$.
\end{proof}

It should be possible to give an alternative proof of Corollary~\ref{coro:Silver_dichotomy_Delta13} using the forceless proof of the dichotomy of chromatic numbers of graphs in \cite{miller_bsl_2012}, but the author has not checked the details.

\begin{mycorollary}
  \label{coro:Delta13_pwo_computable_in_LT2}
      Assume $\boldDelta{2}$-determinacy. Let $x \in \mathbb{R}$. If $\leq^{*}$ is a $\Delta^1_3(x)$ prewellordering on $\mathbb{R}$ and $A $ is a $\Sigma^1_3(x)$ subset of $\mathbb{R}$, then $\sharpcode{\leq^{*}}$ and $\set{ \wocode{y}_{\leq^{*}}}{y \in A}$ are both in $\admistwo{M_1^{\#}(x)}$ and  $\Delta_1$-definable over $\admistwo{M_1^{\#}(x)}$ from parameters in $\se{T_2,M_1^{\#}(x)}$.
\end{mycorollary}
\begin{proof}
  The equivalence relation $ a \equiv^{*} b \eqiv a \leq^{*}b \leq ^{*} a$ is thin. By Corollary~\ref{coro:Silver_dichotomy_Delta13},  we get a $\Delta^1_3(x)$-function $\pi : \mathbb{R} \to u_{\omega}$ such that $a \equiv^{*} b$ iff $\pi(a) = \pi(b)$. $\pi$ induces a wellordering $<^{**}$ on $\ran(\pi)$ where $\pi(a) <^{**} \pi(b)$ iff $a <^{*} b$. $\sharpcode{\leq^{*}}$   is then the order type of $<^{**}$.  $\ran(\pi)$ and $<^{**}$ are $\Sigma^1_3$, hence $\Pi_1$-definable over $\admistwo{x}$ from $\se{T_2,x}$ by Theorem~\ref{thm:BK-KM}. Put $w = M_1^{\#}(x)$. By \cite[Lemma 14.2]{Q_theory}, $\kappa_3^x < \kappa_3^w$. So $\ran(\pi)$ and $<^{**}$ are $\Delta_1$-definable over $\admistwo{w}$ from $\se{T_2,w}$. By admissibility, $\sharpcode{\leq^{*}}$  is $\Delta_1$-definable in $\admistwo{w}$ from $\se{T_2,w}$.  The part concerning $\set{ \wocode{y}_{\leq^{*}}}{y \in A}$ is similar.
\end{proof}

\begin{remark}
  We do not know if $M_1^{\#}(x)$ can be replaced by $x$ in the conclusion of Corollary~\ref{coro:Delta13_pwo_computable_in_LT2}.
\end{remark}

\subsection{$N$-homogeneous trees}
\label{sec:N-homogeneous-trees}

As this paper and its sequels deal with restricted ultrapowers and ``restricted homogeneous trees''  over and over again, it is convenient to abstract the relevant properties.

A transitive set or class $N$ is \emph{admissibly closed} iff
\begin{displaymath}
\forall M \in N \exists M' \in N (M' \text{ is admissible } \wedge M \in M')
\end{displaymath}
Suppose $N$ is admissibly closed and $X \in N$.
$\nu$ is an \emph{$N$-filter} on $X$ iff there is a filter $\nu^{*}$ on $X$ such that $\nu = \nu^{*} \cap N$. 
An $N$-filter $\nu$ is an \emph{$N$-measure} on $X$ iff $\nu$ is countably complete and for any $A \in \power(X) \cap N$, either $A \in \nu$ or $X \setminus A \in \nu$. If $\nu$ is an $N$-measure on $X$, then $\ult(N,\nu)$ is the ultrapower consisting of equivalence classes of functions $f:X \to N$ that lie in $N$. Denote by $j_{N}^{\nu}: N \to \ult(N,\nu)$ the ultrapower map and $[f]^{\nu}_N$ the $\nu$-equivalence class of $f$ in $\ult(N,\nu)$. The ultrapower is well-defined by admissible closedness of $N$, and is wellfounded by countable completeness of $\nu$. The usual \Los{} proof shows for any transitive $M\in N$ containing $\se{X}$,  for any first order formula $\varphi$,   for any $f_i : X \to M$ that belongs to $N$, $1\leq i \leq n$,
\begin{displaymath}
 j^{\nu}_N (M) \models \varphi([f_1]^{\nu}_N,\ldots,[f_n]^{\nu}_N)
\end{displaymath}
iff
\begin{displaymath}
\text{for $\nu$-a.e.\ $a \in X$, } M \models \varphi(f_1(a), \ldots,f_n(a)).
\end{displaymath}
Suppose $\nu$ is an $N$-measure on $X^n$ and $\mu$ is an $N$-measure on $X^m$, $m\leq n$. 
 $\nu$ \emph{projects to} $\mu$ iff for all $A \subseteq X^m$, $A \in \mu$ iff $\set{\vec{\alpha}}{\vec{\alpha} \res m\in A} \in \nu$.  $\vec{\nu} = (\nu_n)_{n<\omega}$ is a \emph{tower of $N$-measures} on $X$ iff for each $n$, $\nu_n$ is an $N$-measure on $X^n$ and $\nu_n$ projects to $\nu_m$ for all $m<n$.

Suppose $N$ is admissibly closed, $X\in N$, and $\vec{\nu} = (\nu_n)_{n<\omega}$ is a tower of $N$-measures on $X$. This naturally induces factor maps $j_N^{\nu_m,\nu_n}$ from $\ult(N,\nu_m)$ to $\ult(N,\nu_n)$.
We say $\vec{\nu}$ is \emph{close to $N$} iff whenever $(A_n)_{n<\omega}$ is a sequence such that $A_n \in \nu_n \cap N$ for all $n$, there exists $(B_n)_{n<\omega} \in N$ such that $B_n \subseteq A_n$ and $B_n \in \nu_n$ for all $n$. If $\vec{\nu}$ is close to $N$, we say $\vec{\nu}$ is \emph{$N$-countably complete} iff whenever $(A_n)_{n<\omega}$ is a sequence such that $A_n \in \nu_n \cap N$ for all $n$, there exists $(a_n)_{n<\omega}$ such that $(a_1,\ldots,a_n) \in A_n$ for all $n$. The usual homogeneous tree argument shows:

\begin{myproposition}\label{prop:countable-complete-to-wellfounded}
 Suppose $\vec{\nu} = (\nu_n)_{n<\omega}$ is close to $N$. Then $\vec{\nu}$ is $N$-countably complete iff the direct limit of $(j^{\nu_m,\nu_{n}}_{N})_{m<n<\omega}$ is wellfounded.
\end{myproposition}
\begin{proof}
  The new part is to show $N$-countable completeness of $\vec{\nu}$ from wellfoundedness of the direct limit of  $(j^{\nu_m,\nu_{n}}_{N})_{m<n<\omega}$. Given $(A_n)_{n<\omega}$ such that $A_n \in \nu_n \cap N$ for all $n$, suppose towards contradiction that there does not exist $(a_n)_{n<\omega}$ such that $(a_1,\ldots,a_n) \in A_n$ for all $n$. By closedness of $\vec{\nu}$ to $N$, let $(B_n)_{n<\omega} \in N$ such that $B_n \subseteq A_n$ and $B_n \in \nu_n$ for all $n$. The tree $T$ consisting of $(a_1,\ldots,a_n)$ such that $a_i \in B_i$ for all $i$ is wellfounded. The ranking function $f$ of $T$ belongs to $N$ by admissible closedness. From $f$ we can construct $f_n: X^n \to N$ so that $f_n \in N$ and $[f_n]_{\nu_n}>[f_{n+1}]_{\nu_{n+1}}$ as usual, contradicting to wellfoundedness of $(j^{\nu_m,\nu_{n}}_{N})_{m<n<\omega}$.
\end{proof}

An $N$-homogeneous system is a sequence $(\nu_s)_{s \in \omega^{<\omega}}$ such that for any $x \in \mathbb{R}$, $\nu_x =_{DEF} (\nu_{x\res n})_{n<\omega}$ is a tower of $N$-measures which is close to $N$. For $X \in N$, a tree $T$ on $\omega \times X$ is $N$-homogeneous iff there is an $N$-homogeneous system $(\nu_s)_{s \in \omega^{<\omega}}$ such that $T_s \in\nu_s $ for all $s \in \omega^{<\omega}$ and for all $x \in p[T]$, $\nu_x$ is $N$-countably complete. If $T$ is $N$-homogeneous, by Proposition~\ref{prop:countable-complete-to-wellfounded} and standard arguments, $x \in p[T]$ iff the direct limit of $(j^{\nu_{x\res m},{\nu_{x\res n}}}_N)_{m<n<\omega}$ is wellfounded.

\subsection{$L[T_{2n+1}]$ as a mouse}
\label{sec:hod_computation}

We recall the Moschovakis tree $T_{2n+1}$ and Steel's computation of $\mathbb{L}[T_{2n+1}]$.

 Assuming $\boldDelta{2n}$-determinacy, $T_{2n+1}$ is the tree of the Moschovakis $\Pi^1_{2n+1}$-scale on a good universal $\Pi^1_{2n+1}$ set, defined in \cite[Chapter 6]{mos_dst}. 

Assuming $\boldpi{2n+1}$-determinacy, 
$\mathcal{F}_{2n,z}$ is the direct system consisting of nondropping countable iterates of $M_{2n}^{\#}(z)$. $M_{2n,\infty}^{\#}(z)$ is the direct limit of $\mathcal{F}_{2n,z}$.  $M_{2n,\infty}^{-} (z)= M_{2n,\infty}^{\#}(z)|\bolddelta{2n+1}$. $(\mathcal{F}_{2n},\mathcal{M}_{2n,\infty}^{\#}, \mathcal{M}_{2n,\infty}^{-}) = (\mathcal{F}_{2n,0}, \mathcal{M}_{2n,\infty}^{\#}(0), \mathcal{M}_{2n,\infty}^{-}(0))$.
\begin{mytheorem}
  \label{thm:steel}
  Assuming $\boldpi{2n+1}$-determinacy. Assume $z$ is a real.
  \begin{enumerate}
  \item $\bolddelta{2n+1}$ is the least $<\delta_{2n,\infty}^z$-strong cardinal of $M_{2n,\infty}^{\#}(z)$, where $\delta_{2n,\infty}^z$ is the least Woodin cardinal of $M_{2n,\infty}^{\#}(z)$.
  \item $M_{2n,\infty}^{-}(z) = L_{\bolddelta{2n+1}}[T_{2n+1},z]$.
  \end{enumerate} 
\end{mytheorem}

The notations concerning inner model theory follow \cite{steel-handbook}. If $\mathcal{M}$ is a premouse, $o(\mathcal{M})$ denotes $\ord \cap \mathcal{M}$. $\mathcal{M} \inisegeq \mathcal{N}$ means that $M$ is an initial segment of $\mathcal{N}$. 
In Steel \cite{steel_projective_wo_1995}, the level-wise projective complexity associated to mice is discussed in detail. 
In this paper, we find it more convenient to work with $\Pi^1_{n+1}$-iterability rather than $\Pi_n^{HC}$-iterability in \cite{steel_projective_wo_1995}. 

We recall the level-wise complexity of projective mice in \cite{steel_projective_wo_1995}. 
A premouse is by definition $\Pi^1_1$-iterable and also $\Pi^1_1$-iterable above any ordinal in the premouse. 
A countable normal iteration tree $\mathcal{T}$ on a countable premouse is $\Pi^1_{2k+1}$-guided iff for any limit $\lambda \leq \lh(\mathcal{T})$, there is $\xi \leq o( \mathcal{M}_{\lambda}^{\mathcal{T}})$ such that $\mathcal{M}_{\lambda}^{\mathcal{T}}| \xi $ is a $\Pi^1_{2k+1}$-iterable above $\mathcal{T}\res \alpha$ and $\rud(\mathcal{M}(\mathcal{T} \res \alpha))\models$``$\delta(\mathcal{T}\res \alpha)$ is not Woodin''. A countable stack of countable normal iteration trees $\vec{\mathcal{T}}$ is $\Pi^1_{2k+1}$-guided iff every normal component of $\vec{\mathcal{T}}$ is $\Pi^1_{2k+1}$-guided.


A countable premouse $\mathcal{P}$ is $\Pi^1_{2k+2}$-iterable above $\eta \in \mathcal{P}$ iff 
for any $\Pi^1_{2k+1}$-guided stack of normal iteration trees $\vec{\mathcal{T}} = (\mathcal{T}_i)_{i < \alpha}$ on $\mathcal{P}$ with critical points above $\eta$, either
\begin{enumerate}
\item the wellfounded model $\mathcal{M}_{\infty}^{\vec{\mathcal{T}}}$ exists, either as the last model of $\mathcal{T}_{\alpha-1}$ when $\alpha$ is a successor or as the direct limit of $(\mathcal{M}_i^{\mathcal{T}}: i < \alpha)$ when $\alpha$ is a limit. 
\item  $\alpha$ is a successor ordinal, $\lh(\mathcal{T}_{\alpha-1})$ is limit and for any $\mathcal{N}\lengtheneq \mathcal{M}(\mathcal{T}_{\alpha})$ which is sound above $\mathcal{M}(\mathcal{T}_{\alpha})$, projects to $\mathcal{M}(\mathcal{T}_{\alpha})$, has $o(\mathcal{M}(\mathcal{T}_{\alpha}))$ as a strong cutpoint and is $\Pi^1_{2k+1}$-iterable above $o(\mathcal{M}(\mathcal{T}_{\alpha}))$, there is a cofinal branch $b$ through $\mathcal{T}_{\alpha}$ such that either  $\mathcal{N} \inisegeq \mathcal{M}_b^{\mathcal{T}}$ or $\mathcal{M}_b^{\mathcal{T}} \inisegeq \mathcal{N}$. 
\end{enumerate}
$\Pi^1_{2k+2}$-iterability above $\eta$ is enough to compare countable ($2k+1$)-small premice that project to $\eta$, agree below $\eta$, and have $\eta$ as a strong cutpoint. A countable normal iteration tree $\mathcal{T}$ on a countable premouse is $\Pi^1_{2k+2}$-guided iff for any limit $\lambda \leq \lh(\mathcal{T})$, there is $\xi \leq o( \mathcal{M}_{\lambda}^{\mathcal{T}})$ such that $\mathcal{M}_{\lambda}^{\mathcal{T}}| \xi$ is $\Pi^1_{2k+2}$-iterable above $\delta(\mathcal{T} | \lambda)$ and $\rud(\mathcal{M}(\mathcal{T} \res \alpha))\models$``$\delta(\mathcal{T}\res \alpha)$ is not Woodin''. A countable stack of countable normal iteration trees $\vec{\mathcal{T}}$ is $\Pi^1_{2k+2}$-guided iff every normal component of $\vec{\mathcal{T}}$ is $\Pi^1_{2k+2}$-guided. 

Assume $\boldDelta{2k+2}$-determinacy. $x \in \mathbb{R}$ codes a $\Pi^1_{2k+3}$-iterable mouse above $\eta$ iff $x $ codes a countable ($2k+2$)-small premouse $\mathcal{P}_x$ and $\eta \in \mathcal{P}_x$ such that for any $v \in \mathbb{R}$ coding $\Pi^1_{2k+2}$-guided stack of normal iteration trees $\vec{\mathcal{T}}  = (\mathcal{T}_i)_{i < \alpha}$ on $\mathcal{P}_x$ with critical points above $\eta$, either
\begin{enumerate}
\item the wellfounded model $\mathcal{M}_{\infty}^{\vec{\mathcal{T}}}$ exists, either as the last model of $\mathcal{T}_{\alpha-1}$ when $\alpha$ is a successor or as the direct limit of $(\mathcal{M}_i^{\mathcal{T}}: i < \alpha)$ when $\alpha$ is a limit, and there is $\mathcal{Q} \lengthen \mathcal{M}_{\infty}^{\vec{\mathcal{T}}}$ such that $\mathcal{Q} \in M_{2k+1}^{\#}(x,v)$, $\mathcal{Q}$ is $\Pi^1_{2k+2}$-iterable above $o(\mathcal{M}_{\infty}^{\vec{\mathcal{T}}})$, $\rud(\mathcal{Q}) \models $``there is no Woodin cardinal $\leq o(\mathcal{M}_{\infty}^{\vec{\mathcal{T}}})$'', or
\item  $\alpha$ is a successor ordinal and there is $b \in M_1^{\#}(x,v)$ such that $b$ is a maximal branch through $\mathcal{T}_{\alpha-1}$, and there is $\mathcal{Q}  \lengthen \mathcal{M}_b^{\mathcal{T}_{\alpha-1}}$ such that  $\mathcal{Q} \in M_{2k+1}^{\#}(x,v)$, $\mathcal{Q}$ is $\Pi^1_{2k+2}$-iterable above $o(\mathcal{M}_{b}^{\mathcal{T}_{\alpha-1}})$, $\rud(\mathcal{Q}) \models $``there is no Woodin cardinal $\leq o(\mathcal{M}_{b}^{\mathcal{T}_{\alpha-1}})$''.
\end{enumerate}
$\Pi^1_{2k+3}$-iterability is a $\Pi^1_{2k+3}$ property by restricted quantification  \cite[4D.3]{mos_dst}. 
``countable'' and ``($2k+2$)-small'' are usually omitted from prefixing ``$\Pi^1_{2k+3}$-iterable mouse''. 
Note that $\Pi^1_{2k+3}$-iterable mice are genuinely $(\omega_1,\omega_1)$-iterable.

$\leq_{DJ}$ is the Dodd-Jensen prewellordering on $\Pi^1_{2k+3}$-iterable mice.  $\mathcal{M}\leq_{DJ} \mathcal{N}$ iff $\mathcal{M},\mathcal{N}$ are $\Pi^1_{2k+3}$-iterable mice and in the comparison between $\mathcal{M}$ and $\mathcal{N}$, the main branch on the $\mathcal{M}$-side does not drop. $\mathcal{M}\sim_{DJ} \mathcal{N}$ iff $\mathcal{M} \leq_{DJ} \mathcal{N} \leq_{DJ} \mathcal{M}$. $\mathcal{M} <_{DJ} \mathcal{N}$ iff $\mathcal{M} \leq_{DJ} \mathcal{N} \nleq_{DJ} \mathcal{M} $.
The norm $x \mapsto \wocode{\mathcal{P}_x}_{<_{DJ}}$ for $x$ coding a $\Pi^1_{2k+3}$-iterable mouse $\mathcal{P}_x$ is $\Pi^1_{2k+3}$. For instance,  $(\mathcal{P}_x$ is a $\Pi^1_{2k+3}$-iterable mouse $\wedge (\mathcal{P}_y $ is a $\Pi^1_{2k+3}$-iterable mouse $\to \mathcal{P}_x \leq_{DJ} \mathcal{P}_y ))$ iff $\mathcal{P}_x$ is a $\Pi^1_{2k+3}$-iterable mouse and for any $\Pi^1_{2k+2}$-guided normal iteration trees $\mathcal{T}, \mathcal{U}$ on $\mathcal{P}_x, \mathcal{P}_y$ respectively, if $\mathcal{T},\mathcal{U}$ have the common last model $\mathcal{Q}$ and the main branch of $\mathcal{T}$ drops, then the main branch of $\mathcal{U}$ also drops. 

If $\mathcal{N}$ is a $\Pi^1_{2k+3}$-iterable mouse, then $\mathcal{I}_{\mathcal{N}}$ is the direct system consisting of countable nondropping iterates of $\mathcal{N}$, and $\mathcal{N}_{\infty}$ is the direct limit of $\mathcal{I}_{\mathcal{N}}$, $\pi_{\mathcal{N},\infty}: \mathcal{N} \to \mathcal{N}_{\infty} $ is the direct limit map. $o(\mathcal{N}_{\infty})< \bolddelta{2k+3}$ as it is the length of a $\boldDelta{2k+3}$-prewellordering. 

For a real $z$, all the iterability notions relativize to $z$-mice. $<_{DJ(z)}$ is the Dodd-Jensen prewellordering on $\Pi^1_{2k+3}$-iterable $z$-mice.




\subsection{Kunen's analysis on subsets of $u_\omega$}
\label{sec:kunens-analys-subs}

Kunen's $\boldDelta{3}$-coding of subsets of $u_{\omega}$ under AD has an effective version under $\boldDelta{3}$-determinacy. 
\begin{mytheorem}[Kunen \cite{sol_delta13coding}]
  \label{thm:Delta13_coding}
Assume $\boldDelta{3}$-determinacy.  There is  $\Delta^1_3$ set $X \subseteq \mathbb{R} \times u_{\omega}$  such that $\set{X_v}{v \in \mathbb{R}} = \power(u_{\omega}) \cap \admistwobold$. Here $X_v = \set{\alpha < u_{\omega}}{ (x, \alpha) \in X}$. 
\end{mytheorem}

The proof in \cite{sol_delta13coding} generalizes easily. 
The only difference is that instead of taking a surjection $h : \mathbb{R} \to \power(u_{\omega})$ in \cite[Lemma 3.7]{sol_delta13coding} under AD by Moschovakis Coding Lemma, we take a surjection $h : \mathbb{R} \to \power(u_{\omega}) \cap \mathbb{L}_{\bolddelta{3}}[T_2]$, where $G$ is a universal $\Pi^1_3$ subset of $\mathbb{R} \times u_{\omega}$ and $h(z) = G_z = \set{\alpha}{(z,\alpha) \in G}$. Surjectivity of $h$ follows from the fact that every subset of $u_{\omega}$ in $\mathbb{L}_{\bolddelta{3}}[T_2]$ is $\boldpi{3}$. In fact, every subset of $u_{\omega}$ in $\mathbb{L}_{\bolddelta{3}}[T_2]$ is $\boldDelta{3}$. The critical step in the proof of Theorem~\ref{thm:Delta13_coding} corresponds to  \cite[Lemma 3.7]{sol_delta13coding}. This step works under $\boldDelta{3}$-Turing determinacy. This is why 
$\boldDelta{3}$-determinacy is an assumption in Theorem~\ref{thm:Delta13_coding}. We don't know if it can be weakened to $\boldDelta{2}$-determinacy. 


\section{More on the level-1 analysis}
\label{sec:level-1-embedding}

We present the usual arguments of Martin's proof of $\boldpi{1}$-determinacy  in a form that conveniently generalizes to higher levels. 

\subsection{The tree $S_1$, level-1 description analysis}
\label{sec:level-1-tree}
We are working under $\text{ZF}+\text{DC}$. 

The technical definition of  \emph{tree of uniform cofinalities} is extracted from \cite{kechris_homo_tree_cabal}, defined in \cite{cabal2_intro_jackson}, and redefined in our paper in a more convenient way. A tree of uniform cofinality pinpoints a particular measure that appears in a homogeneity system for a projective set.
A \emph{level-1 tree of uniform cofinalities}, or a \emph{level-1 tree}, is a set $P \subseteq \omega^{<\omega}$ such that: 
\begin{enumerate}
\item $\emptyset \notin A$.
\item If $(i_1,\ldots,i_{k+1}) \in T$, $k\geq 1$, then $(i_1,\ldots,i_k) \in T$ and for every $j < i_{k+1}$, $(i_1,\ldots,i_k,j) \in T$.
\end{enumerate}
Any countable linear ordering is isomorphic to $<_{BK} \res P$ for some level-1 tree $P$. If $P,P'$ are finite level-1 trees, $s \notin P$, $P'=P \cup \se{s}$, then the $<_{BK} \res P'$-predecessor of $s^{-}$ is $s$. Level-1 trees  are just convenient representations of countable linear orderings and their extensions. 

 A level-1 tree $P$ is said to be \emph{regular} iff $(1) \notin P$. In other words, when $P$ is regular and $P \neq \emptyset$, $(0)$ must be the $<_{BK}$-maximal node of $P$.

The \emph{ordinal representation} of $P$ is
\begin{displaymath}
  \rep(P) = \set{(p)}{p \in P} \cup \set{(p,n)}{p \in P, n<\omega}.
\end{displaymath}
$\rep(P)$ is endowed with the ordering
\begin{displaymath}
  <^P = <_{BK} \res \rep(P). 
\end{displaymath}
Thus, for $p \in P$, $(p)$ is the $<^P$-supremum of $(p,n)$ for $n<\omega$. If $B \subseteq \omega_1$ is in $\mathbb{L}$, let $B ^{P \uparrow}$ the set of functions 
 $f : \rep(P) \to B$ which are continuous, order preserving (with respect to $<^P$ and $<$) and belong to $\mathbb{L}$.
 If $f \in \omega_1^{P \uparrow}$, let
\begin{displaymath}
  [f]^P = ([f]^P_p)_{p \in P},
\end{displaymath}
where  $[f]^P_p = f((p))$ for $p \in P$. 
 Let $[B]^{P\uparrow}= \set{[f]^P}{f \in B^{P \uparrow}}$. 
 $P$ is said to be  \emph{$\Pi^1_1$-wellfounded}   iff  $P \cup \se{\emptyset}$ is a wellfounded tree, or equivalently, $<^P$ is a wellordering. $\Pi^1_1$-wellfoundedness of a level-1 tree is a $\boldpi{1}$ property in the real coding the tree. 
A tuple $\vec{\alpha}=({\alpha}_p)_{p \in P}$ is said to \emph{respect} $P$ iff  $\vec{\alpha}\in [\omega_1]^{P \uparrow}$.  
In other words,  each ${\alpha}_p$ is a countable limit ordinal, and the map $p \mapsto {\alpha}_p$ is an isomorphism between $(P; <_{BK} \res P)$ and $(\se{{\alpha}_p: p\in P}; <)$. In particular, when $P$ is regular, $P \neq \emptyset$ and $\vec{\alpha}$ respects $P$, then ${\alpha}_{(0)}> {\alpha}_p$ whenever  $p \in {P} \setminus \se{(0)}$. 


%

A \emph{finite level-1 tower}  is  a tuple $(P_i)_{i \leq n}$ such that  $ n<\omega$,  $P_i$ is a  level-1 tree of cardinality $i$ for any $ i $, and $i<j \to P_i\subseteq P_j$. An \emph{infinite level-1 tower} is $(P_i)_{i < \omega}$ such that $(P_i)_{i \leq n}$ is a finite level-1 tower for any $n<\omega$. 
A \emph{level-1 system} is a sequence $\vec{P} = (P_s)_{s \in \omega^{<\omega}}$ such that for each $s \in \omega^{<\omega}$, $(P_{s\res i})_{i<\lh(s)}$ is a finite level-1 tower. $\vec{P}$ is \emph{regular} iff each  $P_s$ is regular.
Associated to a $\boldpi{1}$ set $A$ we can assign a regular level-1 system  $(P_s)_{s \in \omega^{<\omega}}$ so that   $x\in A$ iff the  infinite  regular level-1 tree $ {P}_x \DEF  \cup_{n<\omega} {P}_{x \res n}$ is $\Pi^1_1$-wellfounded. If $A$ is lightface $\Pi^1_1$, then $(P_s)_{s \in \omega^{<\omega}}$   can be picked effective.  

\begin{mydefinition}
  \label{def:S_1}
$S_1$ is the tree on $V_{\omega} \times \omega_1$ such that $(\emptyset,\emptyset) \in S_1$ and a nonempty node
  \begin{displaymath}
    (\vec{P}, \vec{\alpha} )  = ( (P_i)_{ i\leq n}, (\alpha_i)_{ i \leq n}  ) \in S_1
  \end{displaymath}
iff  $(P_i)_{ i \leq n}$ is a finite regular level-1 tower and putting $p_i \in P_{i+1}\setminus P_i$,  $\beta_{p_i} = \alpha_{i}$, then  $(\beta_p)_{p \in P_n}$  respects $P_n$.
\end{mydefinition}
Since every tree occurring in $S_1$ is regular, for a nonempty node $(\vec{P},\vec{\alpha}) \in S_1$, we must have $\alpha_0 > \max(\alpha_1,\ldots,\alpha_n)$.  

$S_1$ projects to the universal $\Pi^1_1$ set:
 \begin{displaymath}
   p[S_1] = \set{\vec{P}}{\vec{P} \text{ is a $\Pi^1_1$-wellfounded regular level-1 tower} }.
 \end{displaymath}
The (non-regular) $\omega_1$-scale associated to $S_1$ is $\Pi^1_1$.

\begin{mydefinition}
  \label{def:factoring_1}
 \begin{enumerate}
 \item   Suppose $P$ is a level-1 tree. 
  The set of  \emph{$P$-descriptions} is $\desc(P) \DEF P \cup \se{\emptyset}$. The \emph{constant $P$-description} is $\emptyset$. 
 \item $p \prec p'$ iff $p,p' \in \desc(P)$  and $p <_{BK}p'$. 
 \item Suppose $P,W$ are level-1 trees. A function $\sigma : P \cup \se{\emptyset} \to W \cup \se{\emptyset}$ is said to \emph{factor} $(P,W)$ 
 iff  $\sigma(\emptyset) = \emptyset$ and   ${\sigma}$ preserves the $<_{BK}$-order. (${\sigma}$ does not necessarily preserve the tree order.)
\item Suppose $P$ is a level-1 tree.  $\sigma$  factors $(P,*)$ iff $\sigma$ factors $(P,W)$ for some level-1 tree $W$. 
 \end{enumerate}
\end{mydefinition}

Suppose $P,W$ are $\Pi^1_1$-wellfounded. Then $\ot(<^P) \leq \ot(<^W)$ is equivalent to ``$\exists \sigma~(\sigma$ factors $(P,W))$''.  $\ot(<^P) < \ot(<^W)$ is equivalent to ``$\exists \sigma\exists w \in W~(\sigma$ factors $(P,W) \wedge \forall p \in {P}~\sigma(p) \prec^W w)$''. The higher level analog of this simple fact  will be established in the third paper of this series, which will be an ingredient in the axiomatization of $0^{3\#}$.

If $\sigma : \se{1,\ldots,n} \to \se{1,\ldots,n'}$ is order preserving, recall that
\begin{displaymath}
  j^{\sigma} : \mathbb{L} \to \mathbb{L}
\end{displaymath}
where $j^{\sigma}( \tau^{L[x]} (u_1,\ldots,u_n)) = \tau^{L[x]}(u_{\sigma(1)},\ldots,u_{\sigma(n)}) $. We let
\begin{displaymath}
  j^{\sigma}_{\sup} : u_{n+1} \to u_{n'+1}
\end{displaymath}
where $j^{\sigma}_{\sup} (\beta) = \sup (j^{\sigma})''\beta$. So $j^{\sigma}$ is continuous at $\beta$ iff $j^{\sigma}(\beta)  = j^{\sigma}_{\sup}(\beta)$. The continuity points of $j^{\sigma}$ are characterized by their $\mathbb{L}$-cofinalities:
\begin{mylemma}\label{lem:j_sigma_continuity}
Suppose $\sigma : \se{1,\ldots,n} \to \se{1,\ldots,n'}$ is order preserving, $\beta< u_{n+1}$. Put $\sigma(0) = 0$. Then
  $j^{\sigma}(\beta) \neq j^{\sigma}_{\sup}(\beta)$ iff for some $k$, $\cf^{\mathbb{L}}(\beta) = u_k$ and $\sigma(k)>\sigma(k-1)+1$.  If $\cf^{\mathbb{L}}({\beta}) =u_{k}$ and $\sigma(k)>\sigma(k-1)+1$, then $j^{\sigma}_{\sup}(\beta) = j ^{\sigma_k} \circ j^{\tau_k}_{\sup}(\beta)$, where $\sigma=\sigma_k \circ \tau_k$, $\sigma_k(i) = \sigma(i)$ for $1 \leq i < k$, $\sigma_k(k) = \sigma(k-1)+1$, $\sigma_k(i) = \sigma(i-1)$ for $k < i \leq n+1$. 
\end{mylemma}
The second half of this lemma states that $j^{\sigma}_{\sup}$ acting on points of $\mathbb{L}$-cofinality $u_{k}$ is factored into the ``continuous part'' $j^{\sigma_k}$ and the ``discontinuous part'' $j^{\tau_k}_{\sup}$. 
This simple fact about factoring $j^{\sigma}_{\sup}$ is essentially part of effectivized Kunen's analysis on $u_{\omega}$ in  \cite{sol_delta13coding}.

\subsection{Homogeneity properties of $S_1$}
\label{sec:level-1-sharp}

From now on,  we assume $\boldpi{1}$-determinacy. This is equivalent to  $\forall x \in \mathbb{R} (x^{\#} \text{ exists})$ by Martin \cite{martin_pi11_det} and  Harrington \cite{harr_sharp_det}.

The first $\omega$ uniform indiscernibles  $(u_n)_{n<\omega}$ can be generated by restricted ultrapowers of $\mathbb{L}$. 
Recall that $\mathbb{L} = \bigcup_{x\in \mathbb{R}}L[x]$, which is admissibly closed. Then for every subset $A\subseteq \omega_1$ in $\mathbb{L}$, there is a real $x$ such that $A$ is $\Sigma_1$-definable over $(L_{\omega_1}[x];\in, x)$. 
Let
\begin{displaymath}
\mu_{\mathbb{L}}
\end{displaymath}
be the $\mathbb{L}$-club measure on $\omega_1$, i.e., $A \in \mu_{\mathbb{L}}$ iff $A \in \mathbb{L}$ and $\exists C \in \mathbb{L}~(C \subseteq A \wedge C\text{ is a club in } \omega_1)$. 
 When $P$ is a finite level-1 tree, $ \mu^P$ is the $\mathbb{L}$-measure on $\card(P)$-tuples in $\omega_1$ given by: $ A \in \mu^P $ iff there is $C \in \mu_{\mathbb{L}}$ such that $[C]^{P\uparrow} \subseteq A$. So $\mu^P$ is essentially a variant of  the $\card(P)$-fold product of $\mu_{\mathbb{L}}$, concentrating on tuples whose ordinals are ordered according to the $<_{BK}$-order of $P$. In particular, $\mu^{\emptyset}$ is the principal ultrafilter concentrating on $\se{\emptyset}$.  Put $j^P = j^{\mu^P}_{\mathbb{L}}$, $[f]_{\mu^P} = [f]^{\mu^P}_{\mathbb{L}}$ for $f \in \mathbb{L}$. 
  Standard arguments show that $\ult(\mathbb{L},\mu^P)=\mathbb{L}$, and $j^P (\omega_1) = u_{\card(P)+1}$. For any real $x$, $j^P\res L[x]$ is elementary from $L[x]$ to $L[x]$.

The set of  uncountable $\mathbb{L}$-regular cardinals below $u_{\omega}$ is $\set{u_n}{1 \leq n < \omega}$. The relation ``$\beta = \cf^{\mathbb{L}}(\alpha)$'' is $\Delta^1_3$ (in the sharp codes). Suppose $P$ is a finite level-1 tree, $p \in \desc(P)$. Then
\begin{displaymath}
  \seed_p^P \in  \mathbb{L}
\end{displaymath}
is the element represented modulo $\mu^P$ by the projection map sending  $\vec{\alpha} = (\alpha_{p'})_{p' \in P} $ to $\alpha_p$ if $p \in P$, by the constant function with value $\omega_1$ if $p=\emptyset$. We have $\seed_p^P=u_{ \wocode{p}_{\prec^P}+1}$, where $\wocode{p}_{\prec^P}$ is the $\prec^P$-rank of $p$. In particular, $\seed^P_{\emptyset} = u_{\card(P)+1}=j^P(\omega_1)$. 
For each $p \in P$, $\mu^P$ projects to $\mu_{\mathbb{L}}$ via the map $\vec{\alpha} \mapsto \alpha_p$. 
\begin{displaymath}
  p^P : \mathbb{L} \to \mathbb{L}
\end{displaymath}
is the induced factoring map that sends $j_{\mu_{\mathbb{L}}}(h)(\omega_1)$ to $j^P(h)(\seed_p^P)$. Thus, $p^P$ is the unique map such that for any $z \in \mathbb{R}$, $p^P$ is elementary from $L[z]$ to $L[z]$ and $p^P \circ j_{\mu_{\mathbb{L}}} = j^P$, $p^P(\omega_1) = \seed_p^P$.
 If  $p$ is the $\prec^P$-predecessor of $p'$, then $(p^P)'' u_2$ is a cofinal subset of $\seed_{p'}^P$.
Put
\begin{displaymath}
  \seed^P = (\seed_p^P)_{p \in \desc(P)},
\end{displaymath}
So $p \prec^P p'$ iff $\seed_p^P < \seed_{p'}^P$. Every element in $\mathbb{L}$ is expressible in the form $j^P(h)(\seed^P)$ for some $h \in \mathbb{L}$. 

If $P,P'$ are finite level-1 trees, $P$ is a subtree of $P'$,  then 
 $\mu^{P'}$ projects to $\mu^P$ in the language of Section~\ref{sec:N-homogeneous-trees}, i.e.,  the identity map factors $(P,P')$. Let
\begin{displaymath}
j^{{P},{P}'}  = j^{\mu^P,\mu^{P'}}_{\mathbb{L}} : \mathbb{L}\to \mathbb{L}
\end{displaymath} 
be the factor map given by Section~\ref{sec:N-homogeneous-trees} and let
\begin{displaymath}
  j^{P,P'}_{\sup} : u_{\omega} \to u_{\omega}
\end{displaymath}
be $j^{P,P'}_{\sup} (\alpha) = \sup (j^{P,P'})''\alpha$. Thus, for any real $x$,
\begin{displaymath}
  j^{P,P'}\res L[x] : L[x] \to L[x]
\end{displaymath}
is elementary and
\begin{displaymath}
   j^{P,P'}(\tau^{L[x]}(\seed^{P}_{p_1},\ldots,\seed^{P}_{p_n}))  = \tau^{L[x]}(\seed^{P'}_{p_1},\ldots,\seed^{P'}_{p_n})
\end{displaymath}
for $p_1,\ldots,p_n \in {P}$. 
If $(P_n)_{n<\omega}$ is an infinite level-1 tower, the associated measure tower $(\mu^{P_n})_{ n<\omega}$ is easily seen close to $\mathbb{L}$.

The proof of $\boldpi{1}$-determinacy \cite{martin_pi11_det} shows that
\begin{mytheorem}[Martin]
  \label{thm:martin_pi11_determinacy}
 Assume $\boldpi{1}$-determinacy. Let $(P_n)_{n<\omega}$ be an infinite level-1 tower. The following are equivalent.
  \begin{enumerate}
  \item $(P_n)_{n<\omega}$ is $\Pi^1_1$-wellfounded.
  \item $[\omega_1]^{\bigcup\set{P_n}{n<\omega}\uparrow} \neq \emptyset$.
  \item $(\mu^{P_n})_{n<\omega}$ is $\mathbb{L}$-countably complete.
  \item The direct limit of $(j^{P_m,P_{n}})_{m<n<\omega}$ is wellfounded. 
  \end{enumerate}
\end{mytheorem}

The next two lemmas compute the ``effective uniform cofinality'' of the image of  certain ordinals under level-1 tree factoring maps. 

\begin{mylemma}
  \label{lem:level_2_uniform_cofinality}
  Suppose $(P^{-}, p)$ is a partial level $\leq 1$ tree  whose completion is $P$. $\sigma, \sigma'$ both factor $(P,W)$. $\sigma $ and $\sigma'$ agree on $P^{-}$, $\sigma'(p)$ is the $\prec^W$-predecessor of $\sigma(p)$. Then for any $\beta < j^{P^{-}}(\omega_1)$ such that $\cf^{\mathbb{L}}(\beta) = \seed^{P^{-}}_{p^{-}}$, 
\begin{displaymath}
  \sigma^W \circ j^{P^{-},P}_{\sup}(\beta) =   ({\sigma}')^W_{\sup} \circ j^{P^{-},P} (\beta).
\end{displaymath}
\end{mylemma} 
\begin{proof}
Note that   $\cf^{\mathbb{L}}(j^{P^{-},P}(\beta)) = \seed^P_{p^{-}}$. 
  As in Lemma~\ref{lem:j_sigma_continuity}, $ ({\sigma}')^W_{\sup}$ acting on points of $\mathbb{L}$-cofinality $\seed^P_{p^{-}}$ is decomposed into the discontinuous part $j^{P,P^{+}}_{\sup}$ and the continuous part $({\sigma}^{+})^W$, where $P^{+}$ is the completion of the partial level $\leq 1$ tree $(P,p^{+})$,  $(p^{+})^{-} =p^{-}$, ${\sigma}^{+}$ factors $(P^{+},W)$, ${\sigma}'$ and ${\sigma}^{+}$ agree on $P$, ${\sigma}^{+} (p^{+}) = \sigma(p)$. 
  Let $\iota$ factor $(P,P^{+})$ where $\iota\res P^{-} = \id$, $\iota(p) = p^{+}$. So ${\sigma}^{+} \circ \iota={\sigma}$. By considering the $\seed^{P^{-}}_{p^{-}}$-cofinal sequence in  $\beta$, it is not hard to show that  $ j^{P,P^{+}}_{\sup} \circ j^{P^{-},P} (\beta) = \iota^{P^{+}} \circ j^{P^{-},P}_{\sup} (\beta) $.
Hence,
 \begin{align*}
   ({\sigma}')^W_{\sup} \circ j^{P^{-},P} (\beta) & = ({\sigma}^{+})^W \circ j^{P,P^{+}}_{\sup} \circ j^{P^{-},P} (\beta) \\
   & = ({\sigma}^{+})^W \circ \iota^{P^{+}} \circ j^{P^{-},P}_{\sup} (\beta)\\
   & =\sigma^W \circ j^{P^{-},P}_{\sup} (\beta).
 \end{align*}
\end{proof}

\begin{mylemma}
  \label{lem:level_2_uniform_cofinality_another}
  Suppose $(P, p)$ is a partial level $\leq 1$ tree, $\sigma$ factors $(P,W)$.  Suppose $\beta < j^{P^{-}}(\omega_1)$ and  either
  \begin{enumerate}
  \item $p=-1$, $P^{+} = P$, $\sigma'=\sigma$,  $\cf^{\mathbb{L}}(\beta) = \omega$, or 
  \item $p \neq -1$, $P^{+}$ is the completion of $(P,p)$, $\sigma'$ factors $(P^{+},W)$, $\sigma = \sigma' \res P$,  $\sigma'(p)$ is the $\prec^W$-predecessor of $\sigma(p^{-})$,  $\cf^{\mathbb{L}}(\beta) = \seed^{P}_{p^{-}}$.
  \end{enumerate}
Then
\begin{displaymath}
  \sigma^W (\beta) =   ({\sigma}')^W_{\sup} \circ j^{P,P^{+}} (\beta).
\end{displaymath}
\end{mylemma} 
\begin{proof}
  By commutativity of factoring maps, $\sigma^W(\beta) = (\sigma')^W \circ j^{P,P^{+}}(\beta) $. Note that $\cf^{\mathbb{L}}(j^{P,P^{+}}(\beta)) = \seed^{P^{+}}_{p^{-}}$ when $p\neq -1$,  $\cf^{\mathbb{L}}(j^{P,P^{+}}(\beta)) = \omega$ when $p= -1$. In either case, by Lemma~\ref{lem:j_sigma_continuity}, $(\sigma')^W$ is continuous at $j^{P,P^{+}}(\beta)$.
\end{proof}

\begin{mydefinition}
  \label{def:signature_etc_level_1}
Suppose $W$ is a finite level-1 tree,  $\vec{w} = (w_i)_{i < m}$ is a distinct enumeration of a subset of $W$. 
Suppose $f : [\omega_1]^{W \uparrow} \to \omega_1$ is a function which lies in $\mathbb{L}$. 
\begin{enumerate}[series=f]
\item The \emph{signature} of $f$ is $\vec{w}$ iff there is $C \in \mu_{\mathbb{L}}$ such that
  \begin{enumerate}
  \item for any $\vec{\alpha}, \vec{\beta} \in [C]^{W \uparrow}$, if $(\alpha_{w_0},\ldots,\alpha_{w_{m-1}}) <_{BK} (\beta_{w_0}, \ldots,\beta_{w_{m-1}})$ then $f(\vec{\alpha}) < f(\vec{\beta})$;
  \item for any $\vec{\alpha}, \vec{\beta} \in [C]^{W \uparrow}$, if $(\alpha_{w_0},\ldots,\alpha_{w_{m-1}}) = (\beta_{w_0}, \ldots,\beta_{w_{m-1}})$ then $f(\vec{\alpha}) = f(\vec{\beta})$.
  \end{enumerate}
  In particular, $f$ is constant on a $\mu^W$-measure one set iff the signature of $f$ is $\emptyset$.
\end{enumerate}
Suppose the signature of $f$ is $\vec{w}=(w_i)_{i < m}$.  
\begin{enumerate}[resume*=f]
\item $f$ is \emph{essentially continuous} iff $m>0$ and for $\mu^W$-a.e.\ $\vec{\alpha}$, $f(\vec{\alpha}) = \sup \set{f(\vec{\beta})}{ (\beta_{w_0},\ldots,\beta_{w_{m-1}}) < (\alpha_{w_0},\ldots,\alpha_{w_{m-1}}) }$.  Otherwise, $f$ is \emph{essentially discontinuous}.  
\item Put $[B]^{W \uparrow -1} = [B]^{W \uparrow} \times \omega$. For $w \in \dom(W)$, put $[B]^{W \uparrow w} = \set{(\vec{\beta},\gamma)}{\vec{\beta} \in [B]^{W \uparrow} , \gamma < \beta_w}$.  For $v \in \se{-1} \cup W$, say that the \emph{uniform cofinality} of $f$ is $v$ iff there is $g : [\omega_1]^{W\uparrow v} \to \omega_1$ such that $g \in \mathbb{L}$ and for $\mu^W$-a.e.\ $\vec{\alpha}$, $F(\vec{\alpha}) = \sup \{G(\vec{\alpha}, \beta) : (\vec{\alpha}, \beta) \in [\omega_1]^{W \uparrow v}\}$ and the function $\beta \mapsto G(\vec{\alpha}, \beta)$ is order preserving.  
\end{enumerate}
It is essentially shown in \cite{sol_delta13coding} that $f$ has a unique signature and uniform cofinality. Let $(P_i,p_i)_{i < m} \concat (P_{m})$ be the partial level $\leq 1$ tower of continuous type and let $\sigma$ factor $(P_{m}, W)$ such that $\sigma(p_i) = w_i$ for each $i < m$. Note that $w_i \prec^W w_0$ for $0 < i < m$, so each $P_i$ is indeed a regular level-1 tree. 
\begin{enumerate}[resume*=f]
\item    $\vec{P}= (P_i)_{i \leq m}$ is called the \emph{level-1 tower induced by $f$}.
\item $\sigma$ is called the \emph{factoring map induced by $f$}.
\end{enumerate}
Note that $\sigma \res P_i$ factors $(P_i,W)$ for each $i$. 
\begin{enumerate}[resume*=f]
\item 
The \emph{potential partial level $\leq 1$ tower} induced by $f$ is
  \begin{enumerate}
  \item $(P_m, (p_i)_{i < m})$, if $f$ is essentially continuous;
  \item $(P_m, (p_i)_{i < m} \concat (-1))$, if $f$ is essentially discontinuous and has uniform cofinality $-1$;
  \item $(P_m, (p_i)_{i < m} \concat (p^{+}))$, if $f$ is essentially discontinuous and has uniform cofinality $w_{*} \in W$, $(P_m, p^{+})$ is a partial level $\leq 1$ tree, $\sigma((p^{+})^{-}) = w_{*}$.
  \end{enumerate}
\end{enumerate}
In particular, if $w_{*} \in W$, $f(\vec{\alpha}) = \alpha_{w_{*}}$ is the projection map, then the potential partial level $\leq 1$ tower induced by $f$ is $(\emptyset, (0))$.  \begin{enumerate}[resume*=f]
\item 
The \emph{approximation sequence} of $f$ is $( f_i)_{i \leq m}$ where $\dom(f_i) = [\omega_1]^{P_i \uparrow}$, $f_0$ is the constant function with value $\omega_1$, $f_i(\vec{\alpha}) = \sup \set{f(\vec{\beta})}{\vec{\beta} \in [\omega_1]^{W \uparrow}, (\beta_{w_0},\ldots,\beta_{w_{i-1}})= (\alpha_{p_0},\ldots,\alpha_{p_{i-1}})}$ for $1 \leq i \leq m$. 
\end{enumerate}
In particular, $f_m(\vec{\beta}_{\sigma}) = f(\vec{\beta})$ for $\mu^W$-a.e.\ $\vec{\beta}$.
\end{mydefinition}

Note that all the relevant properties of $f$ depend only on the value of $f$ on a $\mu^W$-measure one set. We will thus be free to say the signature, etc.\ of $f$ when $f$ is defined on a $\mu^W$-measure one set. 

\begin{mydefinition}
  \label{def:ordinal_long_division}
  Suppose $\omega_1 \leq \beta < u_{\omega}$ is a limit ordinal. Suppose $W$ is a finite level-1 tree,  $\beta = [f]_{\mu^W} <  u_{\card(W)+1}$,  the signature of $f$ is $(w_i)_{i < m}$, the approximation sequence of $f$ is $(f_i)_{ i \leq  m}$, the level-1 tower induced by $f$ is $(P_i)_{i \leq m}$, the factoring map induced by $f$ is $\sigma$.  
Then:
\begin{enumerate}
\item The \emph{signature} of $\beta$ is $(\seed^W_{w_i})_{i < m}$.
\item The \emph{approximation sequence} of $\beta$ is $( [f_i]_{\mu^{P_i}})_{i \leq m}$.
\item $\beta$ is \emph{essentially continuous} iff $f$ is essentially continuous.
\item The \emph{uniform cofinality} of $\beta$ is $\omega$ if $f$ has uniform cofinality $-1$. $\seed^W_{w_{*}}$ if $f$ has uniform cofinality $w_{*} \in W \cup \se{\emptyset}$.
\item The \emph{potential partial level $\leq 1$ tower induced by $\beta$} is the potential partial level $\leq 1$ tower induced by $f$.
\end{enumerate}
\end{mydefinition}

The uniform cofinality of $\beta$ is exactly $\cf^{\mathbb{L}}(\beta)$. 
The  signature, approximation sequence and essential continuity of $\beta$ are independent of the choice of $(W,f)$ in Definition~\ref{def:ordinal_long_division}, and moreover $\Delta^1_3$ in $\beta$ uniformly. 


\subsection{The tree $S_2$}
\label{sec:level-2-kechris_tree}

In this section, we redefine the tree $S_2$ introduced in  \cite[Section 2]{kechris_homo_tree_cabal} in the language of trees of uniform cofinalities in \cite{cabal2_intro_jackson}.

A \emph{partial level $\leq 1$ tree} is a pair $(P,t)$ such that $P$ is a finite regular level-1 tree, and either
\begin{enumerate}
\item   $ t\notin {P} \wedge {P} \cup \se{t} $ is a regular level-1 tree, or
\item $P\neq \emptyset$, $t =-1$. 
\end{enumerate}
$-1$ is regarded as the ``level-0'' component, hence the name ``level $\leq 1$''. 
$(P ,t)$ is of degree $0$ if $t = -1$, of degree 1 otherwise. 
We put $\dom(P,t) = P \cup \se{t}$. 
$\vec{\alpha} = (\alpha_s)_{s \in P \cup \se{ t}}$ \emph{respects} $(P,t)$ iff $\vec{\alpha} \res P$ respects $P$ and $t=-1 \to \alpha_{t}<\omega$, $t \neq -1 \to \vec{\alpha}$ respects $P \cup \se{t}$.
The \emph{cardinality} of $(P, t)$ is  $\card(P, t) = \card(P)+1 $.  The unique partial level $\leq 1$ tree of cardinality 1 is $(\emptyset, (0)) $. 
If $(P,t)$ is of degree 1, its  \emph{completion} is ${P} \cup \se{t}$. $(P,-1)$ has no completion. 
 $(P, t)$ is a \emph{partial subtree} of $P'$ iff the completion of $(P, t)$ exists and is a subtree of $P'$. 

A \emph{partial level $\leq 1$ tower of discontinuous type} is a nonempty finite sequence $(\vec{P},\vec{p})  =  (P_i,p_i)_{i \leq k}$ such that $\card(P_0,p_0) = 1$,
each  $(P_i,p_i)$ is a partial level $\leq 1$ tree, and $P_{i+1}$ is the completion of $(P_i,p_i)$. 
A \emph{partial level $\leq 1$ tower of continuous type} is $(P_i,p_i)_{i < k} \concat (P_{*})$ such that  either $k=0 \wedge P_{*} = \emptyset$ or $(P_i,p_i)_{i < k}$ is a partial level $\leq 1$ tower of discontinuous type $\wedge P_{*}$ is the completion of $(P_{k-1},p_{k-1})$. 
For notational convenience, the information of a partial level $\leq 1$ tower is compressed into a potential partial level $\leq 1$ tower. 
We say a \emph{potential partial level $\leq 1$ tower} is $(P_{*},\vec{p})=(P_{*},(p_i)_{i < \lh(\vec{p})})$ such that for some level-1 tower $\vec{P} = (P_i)_{i \leq  k}$, either $P_{*} = P_k \wedge (\vec{P},\vec{p})$ is a partial level $\leq 1$ tower of discontinuous type or $(\vec{P},\vec{p}) \concat (P_{*})$ is a partial level $\leq 1$ tower of continuous type. 
 If $(P_{*} , (p_i)_{i \leq k})$ is a potential partial level $\leq 1$ tower of discontinuous type, its  \emph{completion} is the completion of $(P_{*}, p_k)$. 

Clearly,  a potential partial level $\leq 1$ tower $(P_{*},\vec{p})$ is of continuous type iff $\card(P_{*}) = \lh(\vec{p})$, of discontinuous type iff $\card(P_{*}) = \lh(\vec{p})-1$.

A \emph{tree of level-1 trees} is a tree $T$ on $\omega^{<\omega}$ (i.e., $T \subseteq (\omega^{<\omega})^{<\omega}$ and closed under $\subseteq$) and such that for any $s \in T$, $\set{a \in \omega^{<\omega}}{ s \concat (a) \in T}$ is a level-1 tree.

A \emph{level-2 tree of uniform cofinalities}, or \emph{level-2 tree}, is a function ${Q}$ such that $\dom(Q)$ is a tree of level-1 trees, $\emptyset \in \dom(Q)$ and 
 for any $q \in \dom(Q)$, $(Q (q \res l))_{l \leq  \lh(q)}$ is a partial level $\leq 1$ tower of discontinuous type. In particular, $Q(\emptyset) = (\emptyset, (0))$. 

We denote  $Q(q) = (Q_{\tree}(q), Q_{\node}(q))$ and $Q[q] = (Q_{\tree}(q),  (Q_{\node}(q\res l))_{l \leq \lh(q)})$. So $Q[q]$ is a potential partial level $\leq 1$ tower of discontinuous type. Denote $Q\se{q} = \set{a \in \omega^{<\omega}}{q \concat (a) \in \dom(Q)}$, which is a level-1 tree. 
The \emph{cardinality} of $Q$ is $\card(Q) = \card(\dom(Q))$. $\card(Q)$ could be finite or $\aleph_0$. 

For  $Q$  a level-2 tree, Let
\begin{displaymath}
\bardom(Q) = \dom(Q) \cup \set{q \concat (-1)}{q \in \dom(Q) }.
\end{displaymath}
Here $-1$ is a distinguished element which is $<_{BK}$-smaller than any node in $\omega^{<\omega}$. So $<_{BK} \res \bardom(Q)$ extends $<_{BK} \res \dom(Q)$ where $q \concat (-1)$ comes before any $q \concat (s) \in \dom(Q)$. 
If $q \neq \emptyset$, denote $Q\se{q,-} =\se{q^{-} \concat (-1)} \cup \set{ q^{-} \concat (a)}{Q_{\tree}(q^{-} \concat (a) )= Q_{\tree}(q) \wedge a <_{BK} q (\lh(q)-1)}$,  
$Q\se{q,+} = \se{q^{-}} \cup \set{ q^{-} \concat (a)}{Q_{\tree}(q^{-} \concat (a)) = Q_{\tree}(q) \wedge a >_{BK} q (\lh(q)-1)}$.
For $q \in \bardom(Q)$,  $q$ is \emph{of discontinuous type} if $q \in \dom(Q)$; $q$ is \emph{of continuous type} if $q \in \bardom(Q) \setminus \dom(Q)$. In particular, $\se{\emptyset,(-1)} \subseteq \bardom(Q)$. 
Put $Q[q \concat (-1)] = (P, (Q_{\node}(q \res l))_{l \leq \lh(q)})$, where $P$ is the completion of $Q(q)$. So $Q[q \concat (-1)]$ is a potential partial level $\leq 1$ tower of continuous type.

\begin{mydefinition}
\label{def:Q_desc}
Suppose $Q$ is a level-2 tree. 
A \emph{$Q$-description} is a triple
\begin{displaymath}
\mathbf{q} = (q,P, \vec{p})
\end{displaymath}
such that $q \in \bardom(Q)$ and  $(P,\vec{p}) = Q[q]$. 
 $\desc(Q)$ is the set of $Q$-descriptions.  A $Q$-description $(q,P, \vec{p})$ is of \emph{(dis-)continuous type} iff $q$ is of (dis-)continuous type. The \emph{constant $Q$-description} is $(\emptyset,\emptyset,(0))$. 
\end{mydefinition}

\begin{mydefinition}
\label{def:extended_Q_desc}
  Suppose $Q$ is a level $\leq 2$ tree.  An \emph{extended $Q$-description} is either a $Q$-description or of the form $(2, (q,P,\vec{p}))$ such that $(2, (q \concat (-1), P, \vec{p}))$ is a $Q$-description of continuous type.  $\exdesc(Q)$ is the set of extended $Q$-descriptions. $(d,\mathbf{q}) \in \exdesc(Q)$ is \emph{regular} iff either $(d,\mathbf{q}) \in \desc(Q)$ of discontinuous type or $(d, \mathbf{q}) \notin \desc(Q)$.
\end{mydefinition}

If $\mathbf{q}= (q,P,\vec{p}) \in \desc(Q)$ is of discontinuous type, put $\mathbf{q}\concat (-1) = (q\concat (-1), P^{+}, \vec{p})$ where $P^{+}$ is the completion of $(P,\vec{p})$. 
If $\vec{\alpha} = ({\alpha}_p)_{p \in N}$ is a tuple indexed by $N$, $q \in \bardom(Q)$, $\dom(Q(q^{-})) \subseteq N$ if $q\neq \emptyset$, 
we put
\begin{displaymath}
  \vec{\alpha} \oplus_Q q =  (  \alpha_{p_{0}}, q(0), \ldots, {\alpha}_{p_{\lh(q)-1}}, q(\lh(q)-1)),
\end{displaymath}
where $p_i = Q_{\node}(q \res i)$. 

The ordinal representation of $Q$ is the set
\begin{align*}
\rep(Q) =& \{ \vec{\alpha} \oplus_Q q : q \in \dom(Q) ,  \vec{\alpha} \text{ respects } Q_{\tree}(q)  \} \\
  & \cup \set{\vec{\alpha} \oplus_Q q\concat (-1)}{q \in \dom(Q), \vec{\alpha} \text{ respects } Q(q)}.
\end{align*}
$\rep(Q)$ is endowed with the $<_{BK}$ ordering:
\begin{displaymath}
  <^Q = <_{BK} \res \rep(Q).
\end{displaymath}
 Thus,  the $<^Q$-greatest element is $\emptyset = \emptyset \oplus_{Q } \emptyset$, and the set $\{(\beta,-1):\beta<\omega_1\}$ is $<^Q$-cofinal below $\emptyset$. In general, if $q \in \dom(Q)$ and $\vec{\alpha}$ respects $Q_{\tree}(q)$ and every entry of $\vec{\alpha}$ is additively closed, then $\vec{\alpha}\oplus_Q q$ is the $<^Q$-sup of $\vec{\alpha}\concat (\beta) \oplus q \concat (-1) \in \rep(Q)$.
The fact that $(0)$ is the $<_{BK}$-maximum node of a nonempty regular level-1 tree implies that if $(q,P) \in \desc({Q})$, $q \neq \emptyset$,  $ (\alpha_p)_{p \in P}$ respects ${P}$, then $\alpha_{(0)}$ is bigger than $\alpha_p$  for any $p \in {P} \setminus \se{(0)}$. Hence, when $Q$ is finite, $<^Q$ has order type $\omega_1+1$. If $B\in \mathbb{L}$ is a subset of $\omega_1$, we put
\begin{displaymath}
  f \in B^{Q \uparrow}
\end{displaymath}
iff $f \in \mathbb{L}$ is an order preserving, continuous function from $\rep(Q)$ to $B \cup \se{\omega_1}$. If $f \in B^{Q \uparrow}$, for each $q \in \dom(Q)$, letting $P_q = Q_{\tree}(q)$,   $f_q$ is the function on $[\omega_1]^{P_q \uparrow}$ that sends $\vec{\alpha} $ to $f(\vec{\alpha} \oplus_Q q)$, and 
\begin{displaymath}
  [f]^Q = ([f]^Q_q)_{q \in \dom(Q)}
\end{displaymath}
where $[f]^Q_q = [f_q]_{\mu^{P_q}}$. 

A \emph{level $\leq 2$ tree} is a pair $Q = (\comp{1}{Q},\comp{2}{Q})$ such that $\comp{d}{Q}$ is a level-$d$ tree for $d \in \se{1,2}$. Its \emph{cardinality} is $\card(Q) =\sum_d\card(\comp{d}{Q})$. We follow the convention that $\comp{d}{Q}$ always stands for the level-$d$ component of a level $\leq 2$ tree $Q$. 
$Q$ is a \emph{level $\leq 2$ subtree} of $Q'$ iff $\comp{d}{Q}$ is a level-$d$ subtree of $\comp{d}{Q}'$ for $d \in \se{1,2}$.
$\rep(Q) = \bigcup_{d}( \se{d} \times \rep(\comp{d}{Q}))$. $<^Q = <_{BK} \res \rep(Q)$. So $<^Q$ is essentially the concatenation of  $<^{\comp{1}{Q}}$ and $<^{\comp{2}{Q}}$. $\dom(Q) = \bigcup_{d}( \se{d} \times \dom(\comp{d}{Q})) $, $\bardom(Q) =\bigcup_{d}( \se{d} \times \bardom(\comp{d}{Q})) $, where $\bardom(\comp{1}{Q}) = \dom(\comp{1}{Q}) = \comp{1}{Q}$.   $\desc(Q) = \bigcup_{d}( \se{d} \times \desc(\comp{d}{Q}))$ is the set of $Q$-descriptions. $(d,\mathbf{q}) \in \desc(Q)$ is \emph{of continuous type} iff $d=2$ and $\mathbf{q}$ is of continuous type; otherwise, $(d, \mathbf{q})$ is \emph{of discontinuous type}. 
 $Q$ is \emph{$\Pi^1_2$-wellfounded} iff $\comp{1}{Q}$ is $\Pi^1_1$-wellfounded and $\comp{2}{Q}$ is $\Pi^1_2$-wellfounded. By virtue of the Brower-Kleene ordering, the next proposition is a corollary of Theorem~\ref{thm:martin_pi11_determinacy}. 

\begin{myproposition}
  \label{prop:level_2_tower_wellfounded_equivalence}
Let $Q$ be a level $\leq 2$ tree. Then $Q$ is $\Pi^1_2$-wellfounded iff $<^Q$ is a wellordering on $\rep(Q)$.
\end{myproposition}

As a corollary, if $Q$ is $\Pi^1_2$-wellfounded, then $\ot(<^Q) =  \omega_1+1$.

If $f$ is a function on $\rep(Q)$, let $\comp{d}{f}$ be the function on $\rep(\comp{d}{Q})$ that sends $v$ to $f(d,v)$. 
If $B\in \mathbb{L}$ is a subset of $\omega_1$, we put
\begin{displaymath}
  f \in B^{Q \uparrow}
\end{displaymath}
iff $f\in \mathbb{L}$ is an order preserving, continuous function on $\rep(Q)$, and   $\comp{d}{f} \in B^{\comp{d}{Q}\uparrow}$ for $d \in \se{1,2}$. 
 $f$ represents a $ \card(Q)$-tuple of ordinals
\begin{displaymath}
  [f]^{Q} =  ({}^d[f]^Q_{q})_{(d,q) \in \dom(Q) }
\end{displaymath}
where  $ {}^d[f]^Q_{q} = [\comp{d}{f}]^{\comp{d}{Q}}_q$.
In particular, we must have ${}^2[f]^Q_{\emptyset} = \omega_1$. 
Let
\begin{displaymath}
  [B]^{Q\uparrow} = \set{[f]^{Q}}{ f \in B^{Q \uparrow} }.
\end{displaymath}

Suppose  $(2,\mathbf{q})=(2,(q,P, \vec{p})) \in \exdesc(Q)$.  
If $f \in \omega_1^{Q \uparrow}$, $\comp{2}{f}_{\mathbf{q}}$ is the function on $[\omega_1]^{P \uparrow}$  defined as follows:  $\comp{2}{f}_{\mathbf{q}} = \comp{2}{f}_q$ if $(2,\mathbf{q})  \in  \desc( Q)$; $\comp{2}{f}_{\mathbf{q}} (\vec{\alpha}) =  \comp{2}{f}_{q}(\vec{\alpha} \res \comp{2}{Q}_{\tree}(q))$ if $(2, \mathbf{q}) \notin \desc(Q)$.   
If $\vec{\beta} = (\comp{d}{\beta}_q)_{(d,q) \in \dom(Q)} \in [\omega_1]^{Q \uparrow}$, we define $\comp{d}{\beta}_{\mathbf{q}}$ for  $(d,\mathbf{q}) \in \exdesc(Q) $: if $d=2$, $\mathbf{q} = (q,P,\vec{p})$, put 
$\comp{d}{\beta}_{\mathbf{q}} = [\comp{d}{f}_{\mathbf{q}}]_{\mu^P}$ where $\vec{\beta} = [f]^Q$. 
Clearly, $\comp{2}{\beta}_{\mathbf{q}} = \comp{2}{\beta}_{q}$ if $(2,\mathbf{q}) \in \desc(Q)$ of discontinuous type,  $\comp{2}{\beta}_{\mathbf{q}} = j^{\comp{2}{Q}_{\tree}(q),P}(\comp{2}{\beta}_{q})$ if $(2,\mathbf{q}) \notin \desc(Q)$.  
The next lemma computes the remaining case when $\mathbf{q}\in \desc(Q)$ is  of continuous type, justifying that  $\comp{d}{\beta}_{\mathbf{q}}$ does not depend on the choice of $f$. 

\begin{mylemma}\label{lem:beta_q_unambiguous}
Suppose $Q$ is a level $\leq 2$ tree. 
Suppose $\vec{\beta} = (\comp{d}{\beta}_q)_{(d,q) \in \dom(Q)} \in [\omega_1]^{Q \uparrow}$, $(2,\mathbf{q})=(2,(q,P, \vec{p})) \in \desc({Q})$ is of continuous type, $P^{-} = Q_{\tree}(q^{-})$,  then $\comp{2}{\beta}_{\mathbf{q}} =j^{P^{-},P}_{\sup}(\comp{2}{\beta}_{q^{-}})$.
\end{mylemma}
\begin{proof}
 Let $\vec{\beta} = [f]^Q$, $f \in [\omega_1]^{Q \uparrow}$. Let $v = p_{\lh(q)-1}$.  So $P$ is the  completion of $Q(q^{-}) = (P^{-},v)$. 

Suppose $\gamma = [g]_{\mu^{P^{-}}} < \comp{2}{\beta}_{q^{-}}$, $g \in \mathbb{L}$. So for $\mu^{P^{-}}$-a.e.\ $\vec{\alpha}$, $g(\vec{\alpha}) < \comp{2}{f}_{q^{-}}(\vec{\alpha}) = \sup_{\xi < \alpha_{v^{-}}} \comp{2}{f}_q (\vec{\alpha} \concat (\xi))$, where $\vec{\alpha} \concat (\xi)$ is the extension of $\vec{\alpha}$ whose entry indexed by $v$ is $\xi$. Let $h(\vec{\alpha})$ be the least $\xi < \alpha_{v^{-}}$ such that $g(\vec{\alpha}) < \comp{2}{f}_q(\vec{\alpha} \concat (\xi))$. Then $h \in \mathbb{L}$. By remarkability of level-1 sharps, we get  $C \in \mu_{\mathbb{L}}$ such that for any $\vec{\alpha} \in [C]^{P \uparrow}$, $h(\vec{\alpha} \res\dom(P^{-})) < {\alpha}_{v}$. Hence for any $\vec{\alpha} \in [C]^{P \uparrow}$, $g(\vec{\alpha} \res\dom(P^{-})) < \comp{2}{f}_q(\vec{\alpha})$. Hence $j^{P^{-},P}(\gamma) < \comp{2}{\beta}_q$.

Suppose on the other hand $\gamma = [g]_{\mu^P} < \comp{2}{\beta}_{\mathbf{q}}$. Then for $\mu^P$-a.e.\ $\vec{\alpha}$, $g(\vec{\alpha}) < \comp{2}{f}_q(\vec{\alpha}) = \sup_{\xi < \alpha_v} \comp{2}{f}_q (\vec{\alpha} \res\dom(P^{-}) \concat (\xi)) $. Let $h(\vec{\alpha})$ be the least $\xi < \alpha_v$ such that $g(\vec{\alpha}) < \comp{2}{f}_q (\vec{\alpha} \res\dom(P^{-}) \concat (\xi))$. By remarkability, we get $C \in \mu_{\mathbb{L}}$ and $h' \in \mathbb{L}$ such that for any $\vec{\alpha} \in [C]^{P \uparrow}$, $h(\vec{\alpha}) = h'(\vec{\alpha} \res \set{p}{p \prec^P v})$. Hence, $g(\vec{\alpha} ) < \comp{2}{f}_q (\alpha \res\dom(P^{-}) \concat h'(\vec{\alpha} \res \set{p}{p \prec^P v}))  = j^{P^{-},P}(\eta)$, where $\eta = [\vec{\alpha} \mapsto \comp{2}{f}_q(\vec{\alpha} \concat h'(\vec{\alpha} \res \set{p}{p \prec^{P^{-}} v^{-}}))]_{\mu^{P^{-}}}$. Clearly, $\eta < \comp{2}{\beta}_{q^{-}}$. So $\gamma < j^{P^{-},P}_{\sup}(\comp{2}{\beta}_{q^{-}})$.
\end{proof}

For a level $\leq 2$ tree, The properties of a tuple $[f]^Q$ for $f \in \omega_1^{Q \uparrow}$ are analyzed in \cite{sol_delta13coding,kechris_homo_tree_cabal}.  We restate the key results in the effective context. 
A tuple $\vec{\beta}$ \emph{respects} $Q$ iff $\vec{\beta} = [f]^Q$ for some $f \in \omega_1^{Q \uparrow}$; $\vec{\beta}$ \emph{weakly respects} $Q$ iff $\beta_{\emptyset} = \omega_1$ and for any $q \in \dom(Q) \setminus\se{\emptyset}$, $\beta_q < j^{Q_{\tree}(q^{-}), Q_{\tree}(q)}(\beta_{q^{-}})$. We will need an $\Delta^1_3$ definition of respectivity and weak respectivity. Weak respectability is clearly $\Delta^1_3$ from its definition. 
It is essentially shown in \cite{sol_delta13coding} that respectability is also $\Delta^1_3$.  We restate the relevant definitions in a more applicable fashion.

The next few lemmas are essentially part of effectivized Kunen's analysis \cite{sol_delta13coding} of tuples of ordinals in $u_{\omega}$. The proofs are rather routine
. 

Suppose $E$ is a club in $\omega_1$. For a partial level $\leq 1$ tree $(P,t)$, put $\vec{\alpha} = (\alpha_p)_{p \in P\cup \se{ t} } \in [E]^{(P,t) \uparrow}$ iff $\vec{\alpha}$ respects $(P,t)$, $(\alpha_p)_{p \in P}\in [E]^{P \uparrow}$,  and $t\neq -1 \to \alpha_t \in E$.
For a level $\leq 2$ tree $Q$, put
%
 \begin{align*}
   \rep(\comp{2}{Q}) \res E= &  \set{ \vec{\alpha} \oplus_{\comp{2}{Q}} q}{ q \in \dom(\comp{2}{Q}), \vec{\alpha} \in [E]^{\comp{2}{Q}_{\tree}(q) \uparrow}} \\
   & \cup \set{ \vec{\alpha} \oplus_{\comp{2}{Q}} q \concat (-1)}{q \in \dom(\comp{2}{Q}),  \vec{\alpha} \in [E]^{\comp{2}{Q}(q)\uparrow}}.
 \end{align*}
 Put $\rep(Q)\res E = (\se{1} \times \rep(\comp{1}{Q})) \cup ( \se{2} \times \rep(\comp{2}{Q}) \res E)$. Then $\rep({Q}) \res E$ is a closed subset of $\rep({Q})$ (in the order topology of $<^Q$).

\begin{mylemma}\label{lem:respect_lv2_measure_one_set}
  Suppose $Q$ is a finite level $\leq 2$ tree,  $C \in \mu_{\mathbb{L}}$ is a club. Then $\vec{\beta} \in [C]^{Q \uparrow}$ iff there exist $f \in \omega_1^{Q \uparrow}$ and $E \in \mu_{\mathbb{L}}$ such that $\vec{\beta} = [f]^Q$ and  for any $q \in \comp{1}{Q}$, $\comp{1}{f}(q)$ is a limit point of $C$;
 for any $q \in \dom(\comp{2}{Q})$, for any $\vec{\alpha} \in [E]^{\comp{2}{Q}_{\tree}(q)\uparrow}$, $\comp{2}{f}_q(\vec{\alpha})$ is a limit point of $C$.
\end{mylemma}
\begin{proof}
The nontrivial direction is $\Leftarrow$.  Suppose $f \in \omega_1^{Q \uparrow}$ and  $E \in \mu_{\mathbb{L}}$ are as given. For $q \in \dom(\comp{2}{Q})\setminus\se{\emptyset}$,  let $\comp{2}{Q}(q)= (P_q,p_q)$, and let $q^{*}$ be the $<_{BK}$-maximum of $\comp{2}{Q}\{q,-\}$.  
\begin{myclaim}
  \label{claim:respect_lv2_distance}
  There is $E' \in \mu_{\mathbb{L}}$ such that $E' \subseteq E$ and for any $q \in \dom(\comp{2}{Q})\setminus\se{\emptyset}$, 
for any $\vec{\alpha} \in [E']^{P_q\uparrow}$, if $p_q \neq -1$ then 
  $C \cap (\comp{2}{f}_{q^{*}}(\vec{\alpha}), \comp{2}{f}_q (\vec{\alpha}))$ has order type $\alpha_{p_q^{-}}$.
\end{myclaim}
\begin{proof}[Proof of Claim~\ref{claim:respect_lv2_distance}]
  Otherwise, there is $q \in \dom(\comp{2}{Q}) \setminus\se{\emptyset}$ such that $p_q \neq -1$ and for $\mu^{P_q}$-a.e.\ $\vec{\alpha}$, $C \cap (\comp{2}{f}_{q^{*}} (\vec{\alpha}), \comp{2}{f}_q(\vec{\alpha}))$ has order type smaller than $\alpha_{p_q^{-}}$.
However, by assumption, $C \cap (\comp{2}{f}_{q^{*}} (\vec{\alpha}), \comp{2}{f}_q(\vec{\alpha}))$  is cofinal in $\comp{2}{f}_q(\vec{\alpha})$, and $\comp{2}{f}_{q \concat (-1)}$ witnesses that $\comp{2}{f}_q$ has uniform cofinality $p_q^{-}$. This leads to a function $h \in \mathbb{L}$ where for $\mu^{P_q}$-a.e.\ $\vec{\alpha}$, $h(\vec{\alpha})$ is a cofinal sequence in $\alpha_{p_q^{-}}$ of order type $< \alpha_{p_q^{-}}$. Hence, 
 $\cf^{\mathbb{L}}( \seed^{P_q}_{p_q^{-}}) < \seed^{P_q}_{p_q^{-}}$ by \Los{}, which is absurd. 
\end{proof}
Fix $E'$ as in Claim~\ref{claim:respect_lv2_distance}. 
We are able to define $f' : \rep({Q})\res E' \to C$ such that $f(1, q) = f' (1, q)$ for $q \in \comp{1}{Q}$,
 $f (2, \vec{\alpha} \oplus_{\comp{2}{Q}} q) = f'(2, \vec{\alpha} \oplus_{\comp{2}{Q}} q)$ for $q \in \dom(\comp{2}{Q})\setminus\se{\emptyset}$, $\vec{\alpha} \in [E'']^{P_q\uparrow}$. 
  Let $\theta: \rep(Q) \to \rep(Q) \res E'$ be an order preserving bijection. 
 Let $E'' \in \mu_{\mathbb{L}}$ where $\eta \in E''$ iff $E' \cap \eta$ has order type $\eta$. 
It is easy to see that $\theta\res (\rep(Q) \res E'') $ is the identity map. Define $g = f' \circ \theta$. Then $g \in C^{Q \uparrow}$ and $[g]^Q = [f]^Q$.
\end{proof}

\begin{mylemma}
  \label{lem:level_2_Q_respecting_function_sort}
  Suppose $Q$ is a finite level $\leq 2$ tree, $\comp{2}{Q}(q) = (P_q, p_q)$ for $q \in \dom(Q)$, $E \in \mu_{\mathbb{L}}$ is a club. Suppose $f : \rep(Q) \res E \to \omega_1+1$ satisfies
  \begin{enumerate}
  \item $f \res( \se{1} \times  \rep(\comp{1}{Q}))$ is continuous, order preserving;
  \item if $q \in \dom(\comp{2}{Q})$, then the potential partial level $\leq 1$ tower induced by $\comp{2}{f}_{q}$ is $\comp{2}{Q}[q]$, the approximation sequence of $\comp{2}{f}_q$ is $(\comp{2}{f}_{q \res i})_{i \leq \lh(q)}$,
and the uniform cofinality of $\comp{2}{f}_q$ on $[E]^{P_q\uparrow}$ is witnessed by $\comp{2}{f}_{q \concat (-1)}$, i.e., if $\vec{\alpha} \in [E]^{P_q \uparrow}$, then $\comp{2}{f}_q(\vec{\alpha}) = \sup\{ \comp{2}{f}_{q \concat (-1)}(\vec{\alpha} \concat (\beta)) : \vec{\alpha} \concat (\beta) \in \rep(\comp{2}{Q}) \res E \}$, and the map $\vec{\beta} \mapsto \comp{2}{f}_{q \concat (-1)}(\vec{\alpha} \concat (\beta))$ is continuous, order preserving;
  \item if $a,b \in \comp{2}{Q}\se{q}$ and $a < _{BK} b$, then $[f_{q \concat (a)}]_{\mu^{P_{q \concat (a)}}} < [f_{q \concat (b)}]_{\mu^{P_{q \concat (b)}}}$. 
  \end{enumerate}
Then there is $E' \in \mu_{\mathbb{L}}$ such that $E' \subseteq E$ and $f \res (\rep(Q) \res E')$ is order preserving.
\end{mylemma}
\begin{proof}
We know  by assumption that  for $\mu^{P_q}$-a.e.\ $\vec{\alpha}$, $f_q ( \vec{\alpha}) = \sup \{f_{q \concat (a)}(\vec{\alpha} \concat (\beta)) : \beta < \alpha_{p_q^{-}}\}$. Fix for the moment $q$ such that $p_q \neq -1$. 
For $\vec{\alpha} = (\alpha_p)_{p \in P_q}$, put $\vec{\alpha}^{-} = (\alpha_p)_{p <_{BK}p_k^{-}}$.
By remarkability of (level-1) sharps, there is a function $h\in \mathbb{L}$ and $E'_q \in \mu_{\mathbb{L}}$ such that 
for any $\vec{\alpha}\in [E'_q]^{P_q\uparrow}$, $h(\vec{\alpha}^{-})< \alpha_{p_q^{-}}$ and for any $\beta\in \alpha_{p_q^{-}}\cap E'_q$,  for any $a, b \in \comp{2}{Q}\se{q}$,   $f_{q \concat (a)}(\vec{\alpha}\concat (\beta)) < f_{q \concat (b)} (\vec{\alpha}\concat (h(\vec{\alpha}^{-})))$. Let $\eta \in E''_q$ iff for any $\vec{\alpha} \in [E'_q]^{P_q\uparrow}$, if $\forall p<_{BK} p_k^{-}  ~ \alpha_p < \eta$ then $h(\vec{\alpha}^{-}) < \eta$. Finally,  let $E'' = \bigcap \set{E''_q}{p_q \neq -1}$.  $E''$ works for the lemma.
\end{proof}

\begin{mylemma}
  \label{lem:Q_respecting}
  Suppose that $Q$ is a finite level $\leq 2$ tree and $\vec{\beta}= (\comp{d}{\beta}_q)_{(d,q)\in \dom(Q)}$ is a tuple of ordinals in $u_{\omega}$. Then $\vec{\beta}$ respects $Q$ iff all of the following holds:
  \begin{enumerate}
  \item $(\comp{1}{\beta}_q)_{q \in \comp{1}{Q}} $ respects $\comp{1}{Q}$. 
  \item For any $q \in \dom(\comp{2}{Q})$, the potential partial level $\leq 1$ tower induced by $\comp{2}{\beta}_q$ is $Q[q]$, and the  approximation sequence of $\comp{2}{\beta}_q$ is  $(\comp{2}{\beta}_{q \res l})_{  l \leq \lh(q)}$.
    \item If $a,b\in \comp{2}{Q}\se{q}$ and $a<_{BK}b$ then $\comp{2}{\beta}_{q \concat (a)} < \comp{2}{\beta}_{q \concat (b)}$.
  \end{enumerate}
Moreover, if $C \in \mu_{\mathbb{L}}$ is a club, then $\vec{\beta} \in [C]^{Q\uparrow}$ iff $\vec{\beta}$ respects $Q$ and letting $C'$ be the set of limit points of $C$, then $\comp{1}{\beta}_q \in C'$ for $ q \in \comp{1}{Q}$, $\comp{2}{\beta}_q \in j^{\comp{2}{Q}_{\tree}(q)} (C')$ for $q \in \dom(\comp{2}{Q})$.
\end{mylemma}

\begin{mylemma}
  \label{lem:respecting_level_2_tree_is_Delta13}
  The relation ``$Q$ is a finite level $\leq 2$ tree $\wedge$ $\vec{\beta}$ respects $Q$'' is $\Delta^1_3$.
\end{mylemma}

\begin{mylemma}
  \label{lem:unique_level_2_tree_represent}
  Suppose $Q$ and $Q'$ are  level $\leq 2$ trees with the same domain.  Suppose $\vec{\beta}$ respects both $Q$ and $Q'$. Then $Q = Q'$. 
\end{mylemma}

If $y \in [\dom(Q)]$, let $Q(y) \DEF  \cup_{n<\omega} {{Q}_{\tree}(y \res n)}$ be an infinite level-1 tree. 
$Q$ is \emph{$\Pi^1_2$-wellfounded} iff
\begin{enumerate}
\item $\forall q \in \dom(Q) ~Q\se{q}$ is $\Pi^1_1$-wellfounded,
\item $\forall y \in [\dom(Q)] ~ Q(y)$ is not $\Pi^1_1$-wellfounded.
\end{enumerate}
In particular, finite level-2 trees are $\Pi^1_2$-wellfounded. $\Pi^1_2$-wellfoundedness of a level-$2$ tree is a $\boldpi{2}$ property in the real coding the tree. 

A level-2 tree $Q$ is a called a \emph{subtree} of $Q'$ iff $Q$ is a subfunction of $Q'$.  
A \emph{finite level-2 tower} is a (possibly empty) sequence $(Q_i)_{1 \leq i \leq n}$ such that   $Q_i$ is a level-$2$ tree for $1 \leq i \leq n$, $\card(Q_i) = i$ and $i<j \to Q_i$ is a subtree of $Q_{j}$. 
An \emph{infinite level-2 tower} is a sequence $\vec{Q}=(Q_n)_{1 \leq n<\omega}$ such that for each $n$, $(Q_i)_{ 1 \leq i \leq n}$ is a finite level-2 tower.
%
\emph{A level-2 system} is $(Q_s)_{s \in \omega^{<\omega}}$ such that for each $s$, $(Q_{s\res i})_{1 \leq i<\lh(s)}$ is a finite level-2 tower.
Associated to a $\boldpi{2}$ set $A$ we can assign a level-2 system  $(Q_s)_{s \in \omega^{<\omega}}$ so that  $x\in A$ iff the level-2 tower $ Q_x \DEF  (Q_{x \res n})_{n<\omega}$ is $\Pi^1_2$-wellfounded. If $A$ is lightface $\Pi^1_2$, then $(Q_s)_{s \in \omega^{<\omega}}$   can be picked effective.


In our language, the level-2 tree $S_2$, originally defined in  \cite[Section 2]{kechris_homo_tree_cabal},  takes the following form. 
\begin{mydefinition}
  \label{def:S_2}
Assume $\boldpi{1}$-determinacy. 
  \begin{enumerate}
  \item 
$S_2^{-}$ is the tree on $V_{\omega} \times u_\omega$ such that $(\emptyset,\emptyset) \in S_2^{-}$ and a nonempty node
    \begin{displaymath}
      (\emptyset,\emptyset) \neq  (\vec{Q}, \vec{\alpha} )  = ( (Q_i)_{1 \leq  i\leq n}, (\alpha_i)_{1 \leq  i \leq n}  ) \in S_2^{-}
    \end{displaymath}
    iff $\vec{Q}$ is a finite level-2 tower, and putting $Q_0 = \emptyset$, $q_i \in \dom(Q_{i+1}) \setminus \dom(Q_i)$,  $\beta_{q_i} = \alpha_i$, then  $(\beta_q)_{q \in \dom(Q_n)}$ respects $Q_n$.
  \item   
$S_2$ is the tree on $V_{\omega} \times u_\omega$ such that $(\emptyset,\emptyset) \in S_2^{-}$ and a nonempty node
    \begin{displaymath}
      (\emptyset,\emptyset) \neq  (\vec{Q}, \vec{\alpha} )  = ( (Q_i)_{1 \leq  i\leq n}, (\alpha_i)_{1 \leq  i \leq n}  ) \in S_2^{-}
    \end{displaymath}
    iff $\vec{Q}$ is a finite level-2 tower, and putting $Q_0 = \emptyset$, $q_i \in \dom(Q_{i+1}) \setminus \dom(Q_i)$,  $\beta_{q_i} = \alpha_i$, then  $(\beta_q)_{q \in \dom(Q_n)}$ weakly respects $Q_n$.
  \end{enumerate}
\end{mydefinition}
 By Theorem~\ref{thm:martin_pi11_determinacy},
 \begin{displaymath}
   p[S_2^{-}] = p[S_2] = \set{\vec{Q}}{\bigcup \vec{Q} \text{ is $\Pi^1_2$-wellfounded} }.
 \end{displaymath}
The (non-regular) $u_{\omega}$-scale associated to $S_2$ is $\Delta^1_3$ (cf. \cite{kechris_homo_tree_cabal}).

\section{More on the level-2 analysis}
\label{sec:level-2-embedding}

\subsection{Homogeneity properties of $S_2$}
\label{sec:measures-on-u_omega}
By \cite[Lemma 14.2]{Q_theory}, $\admistwobold$ is admissibly closed. We shall define a system of $\admistwobold$-measures on finite tuples in $u_{\omega}$. This system of $\admistwobold$-measures will witness $S_2$ being $\admistwobold$-homogeneous. Under AD, these $\admistwobold$-measures are total measures induced from the strong partition property on $\omega_1$ (cf. \cite{kechris_homo_tree_cabal}). These measures enable  the Martin-Solovay tree construction of $S_3$ projecting to the universal $\Pi^1_3$ set, to be redefined in Section~\ref{sec:weak-partition-LT3-measures}.  In our situation, we must recast the effective version of the proof of the strong partition property on $\omega_1$. 
Only subsets of $\omega_1$ in $\mathbb{L}$ will be colored, and the coloring must be guided by a level-2 tree $Q$ and a subset $A$ of $ [\omega_1]^{Q \uparrow}$ which lies in $\admistwobold$. 





\begin{mydefinition}
  \label{def:strong_partition_omega_1}
  $\omega_1$ has the \emph{level-2 strong partition property} iff  for every finite level $\leq 2$ tree $Q$, for every $A \in \admistwobold$, there is a club $C \subseteq \omega_1$, $C \in \mathbb{L}$ such that either $[C]^{Q \uparrow} \subseteq A$ or $[C]^{Q \uparrow} \cap A = \emptyset$.
\end{mydefinition}


Martin's proof of the strong partition property on $\omega_1$ under AD carries over in a trivial way. For the reader's convenience, we include a proof.
\begin{mytheorem}[Martin]
  \label{thm:strong_partition_omega_1}
  Assume $\boldDelta{2}$-determinacy. Then  $\omega_1$ has the level-2 strong partition property.
\end{mytheorem}
 \begin{proof}
  We imitate the proof in  \cite[Theorem 28.12]{kanamori}, which builds on partially iterable sharps.

For $x \in \mathbb{R}$, 
A putative $x$-sharp is a remarkable EM blueprint over $x$.  Suppose $x^{*}$ is a putative $x$-sharp. For any ordinal $\alpha$, $\mathcal{M}_{x^{*},\alpha}$ is the EM model built from $x^{*}$ and indiscernibles of order type $\alpha$. The wellfounded part of $\mathcal{M}_{x^{*}, \alpha}$ is transitive. For any limit ordinal $\alpha < \beta$, $\mathcal{M}_{x^{*}, \alpha}$ is a rank initial segment of $\mathcal{M}_{x^{*}, \beta}$. Say that $x^{*}$ is $\alpha$-wellfounded iff $\alpha \in \wfp(\mathcal{M}_{x^{*}, \alpha})$. A \emph{putative sharp code for an increasing function}  is $w = \corner{\gcode{\tau}, x^{*}}$ such that $x^{*}$ is a putative $x$-sharp, $\tau$ is a $\se{\underline{\in}, \underline{x}}$-unary Skolem term for an ordinal and
\begin{displaymath}
``\forall v,v'( (v , v'\in \ord \wedge v < v' )\to  (\tau(v) \in \ord  \wedge \tau(v) < \tau(v')))"
\end{displaymath}
is a true formula in $x^{*}$. The statement ``$\corner{\gcode{\tau}, x^{*}}$ is a putative sharp code for an increasing function, $x^{*}$ is $\alpha$-wellfounded, $r$ codes the order type of $\tau^{\mathcal{M}_{x^{*}, \alpha}}(\alpha)$''  about   $(\corner{\gcode{\tau}, x^{*}},r)$ is
$\boldsigma{1}$ in the code of $\alpha$. In addition, when $x^{*} = x^{\#}$, $\corner{\gcode{\tau}, x^{*}}$ is called a \emph{(true) sharp code for an increasing function}. 
The statement `` $M$ is a putative $v$-sharp for some  $v\in \mathbb{R}$, $\alpha\in \wfp(M_{\infty})$, $r$ codes the order type of $\tau^{M_{\infty}}(\alpha)$'' is a $\boldsigma{1}$ statement about $(M,r)$ in the code of $\alpha$.

To simplify matters, we shall ignore the level-1 component of $Q$ and  assume that $Q$ is a finite level-2 tree. 
Let $A \in \mathbb{L}_{\bolddelta{3}}[T_2]$. 
$A$ is $\boldDelta{3}$ by Theorem~\ref{thm:BK-KM}. Let $B,C \subseteq \mathbb{R}^2$ be $\boldsigma{2}$ such that
\begin{multline*}
  w \text{ codes } ({w}_q)_{q \in \dom(Q)} \in (\WO_{\omega})^{\dom(Q)} \wedge (\sharpcode{{w}_q})_{q \in \dom(Q)} \in A \\
\eqiv \exists z (w,z) \in B \eqiv \neg\exists z (w,z) \in C.
\end{multline*}
We define the game
\begin{displaymath}
H^{Q}(B)
\end{displaymath}
in which I produces $\corner{\gcode{\tau}, x^{*}}$ and II produces $(\corner{\gcode{\sigma}, y^{*}},w,z)$. 
An infinite run  ($\corner{\gcode{\tau}, x^{*}}, \corner{\gcode{\sigma}, y^{*}},w,z)$ is won by Player II iff
 \begin{enumerate}
\item  If $\corner{\gcode{\tau}, x^{*}}$ is a putative sharp code for an increasing function, then so is $\corner{\gcode{\sigma},y^{*}}$. Moreover, for any $\eta < \omega_1$, if
  \begin{displaymath}
    x^{*} \text{ is $\eta$-wellfounded }\wedge \tau^{\mathcal{M}_{x^{*},\eta}}(\eta) \in \wfp(\mathcal{M}_{x^{*},\eta})
  \end{displaymath}
then   
  \begin{displaymath}
    y^{*} \text{ is $\eta$-wellfounded }\wedge \sigma^{\mathcal{M}_{y^{*},\eta}}(\eta) \in \wfp(\mathcal{M}_{y^{*},\eta})  .
  \end{displaymath}
\item If $\corner{\gcode{\tau}, x^{*}} , \corner{\gcode{\sigma}, y^{*}} $ are true sharp codes for increasing functions, $x^{*} = x^{\#}$, $y^{*} = y^{\#}$, then $w$ codes $(w_{q})_{q \in \dom(Q)} \in (\WO_{\omega})^{\dom(Q)}$ and
 for $q \in \dom(Q)$,  letting $g^q$ be defined by
  \begin{displaymath}
    g^{q}(\vec{\alpha}) = \sup_{\beta}
    \left(
    \begin{array}{l}
\tau^{L[x]}(  \theta_Q((\vec{\alpha}\concat \beta) \oplus (q \concat (-1)))),\\
 \sigma^{L[y]}(\theta_Q((\vec{\alpha}\concat \beta) \oplus (q \concat (-1))))
\end{array}
    \right),
  \end{displaymath}
where the sup ranges over $\beta$ such that $ (\vec{\alpha}\concat \beta) \oplus (q \concat(-1))\in \rep(<^{Q})$, $\theta_Q$ is an order preserving bijection from $\rep(Q)$ to $\omega_1+1$, we have
\begin{displaymath}
 \forall q \in \dom(Q) [g^{q}]_{\mu^{Q_{\tree}(q)}}  =  \sharpcode{w_{q}}.
\end{displaymath}
and
\begin{displaymath}
  (w,z) \in B.
\end{displaymath}
 \end{enumerate}

\begin{mylemma}
  \label{lem:martin-spp-proof}
  If Player II has a winning strategy in $H^{Q}(B)$, then there is a club $X \subseteq \omega_1$ such that $X \in \mathbb{L}$ and $[X]^{Q\uparrow} \subseteq A$.
\end{mylemma}
\begin{proof}
  Let $\varphi$ be a winning strategy for Player II in $H^{Q}(B)$. We define a club $X\subseteq \omega_1$ by the  $\boldsigma{1}$-boundedness argument. For $\eta<\omega_1$, let $B_{\eta}$ be the set of $r\in \mathbb{R}$ such that there are $\corner{\gcode{\tau}, x^{*}}$, $\corner{\gcode{\sigma},y^{*}}$, $w$,$z$ and an ordinal $\beta \leq \eta$ such that
  \begin{enumerate}
  \item $\corner{\gcode{\tau}, x^{*}}$ is a putative sharp code for an increasing function, 
  \item $\corner{\gcode{\sigma}, y^{*}}$ is a putative sharp code for an increasing function, 
  \item $\corner{\gcode{\tau}, x^{*}}* \varphi = (\corner{\gcode{\sigma}, y^{*}}, w,z)$,
  \item $\beta \in \wfp(\mathcal{M}_{x^{*},\eta}) \wedge \tau^{M_{x^{*},\eta}}(\beta) \in \wfp(\mathcal{M}_{x^{*},\eta}) \wedge \tau^{M_{x^{*},\eta}}(\beta) \leq \eta$.
  \item $\beta \in \wfp(\mathcal{M}_{y^{*},\eta}) \wedge \sigma^{\mathcal{M}_{y^{*},\eta}}(\beta) $ has order type coded in $r$.
  \end{enumerate}
$B_{\eta}$ is a $\boldsigma{1}$ set in the code of $\eta$. Since $\varphi$ is winning for II, $B_\eta \subseteq WO$. By $\boldsigma{1}$-boundedness, if $X\subseteq \omega_1$ is the club consisting of $\varphi$-admissibles and limits of $\varphi$-admissibles, then for any $\xi \in X$, for any $\eta<\xi$ and $r \in B_{\eta}$, $\|r\|<\xi$. 

We have to show that $[X]^{Q\uparrow} \subseteq A$. That is, for any $f \in X^{Q\uparrow} \cap \mathbb{L}$, $[f]^{Q} \in A$. Pick such an $f$. Let ${x}\in \mathbb{R}$ and $\tau$ be  such that  for any $\vec{\alpha} \oplus q \in \rep({Q})$, 
\begin{displaymath}
  f(\vec{\alpha} \oplus q) = \tau^{L[{x}]} (\theta_Q(\vec{\alpha} \oplus q)).
\end{displaymath}
Feed in $\corner{\gcode\tau ,{x}^{\#}}$ for Player I in $H^{Q}(A;n)$. The response according to $\varphi$ is $ =  (\corner{\gcode\sigma , {y}^{\#}},w,z)*\varphi$. $\corner{\gcode\sigma , {y}^{\#}}$ is a true sharp code for an increasing function.  $w$ codes $(w_q)_{q \in \dom(Q)} \in (\WO_{\omega})^{\dom(Q)}$. 
Let $g^q$ be as in the definition of $H^{Q}(B)$. Thus, $[g^q]_{\mu^{Q_{\tree}(q)}} = \sharpcode{w_q}$ and $(w,z) \in B$. Thus, $(\sharpcode{w_q})_{q \in \dom(Q)} \in A$. 
To finish the proof, we have to see that
\begin{displaymath}
  [f_q]_{\mu^{Q_{\tree}(q)}} = [g^q]_{\mu^{Q_{\tree}(q)}}
\end{displaymath}
for all $q \in \dom(Q)$. 
It suffices to see that whenever $\vec{\alpha} $ respects $Q(q)$,
\begin{displaymath}
  \sup_{\beta} f ( (\vec{\alpha}\concat \beta)\oplus (q \concat (-1))) =  \sup_{\beta}  \left(
    \begin{array}{l}
f ( (\vec{\alpha}\concat \beta)\oplus (q \concat (-1))),    \\
   \sigma^{L[y]}  (\theta_Q ( (\vec{\alpha}\concat \beta)\oplus (q \concat (-1))  ))
 \end{array}
 \right)
\end{displaymath}
 $\leq$ is evident. To get $\geq$, by choice of $X$, for any $\beta$ which is used in the supremum, 
\begin{displaymath}
  \sigma^{L[y]}  (\theta_Q ( (\vec{\alpha}\concat \beta)\oplus (q \concat (-1))  )) < \min(X \setminus ({ f ( (\vec{\alpha}\concat \beta)\oplus (q \concat (-1)))   +1})). 
\end{displaymath}
The right hand side of the above inequality is  $\leq f ( (\vec{\alpha}\concat \beta+1)\oplus (q \concat (-1)))$, as $f$ is $<^{Q}$-order preserving into $X$. 
\end{proof}

Define the game $H^Q(C)$ in the same way. A symmetrical argument gives 
\begin{mylemma}\label{lem:martin-spp-another}
    If Player II has a winning strategy in $H^{Q}(C)$, then there is a club $X \subseteq \omega_1$ such that $X \in \mathbb{L}$ and $[X]^{Q\uparrow} \cap A = \emptyset$.
\end{mylemma}

The games $H^Q(B)$ and $H^Q(C)$ are both  $\boldgameclass{}$, hence determined. It remains to show that II must have a winning strategy in either $H^Q(B)$ or $H^Q(C)$. Suppose otherwise and I has a winning strategy $\varphi_B$ in $H^Q(B)$ and $\varphi_C$ in $H^Q(C)$. We apply the same boundedness argument as in the proof of Lemma~\ref{lem:martin-spp-proof}. Let $X$ be the set of countable $(\varphi_B,\varphi_C)$-admissibles and their limits. Let $f \in [X]^{Q \uparrow}$. If $[f]^Q \in A$, pick $(w,z) \in B$ with $w$ coding $(w_q)_{q \in \dom(Q)}$ and $(\sharpcode{w_q})_{q \in \dom(Q)} = [f]^Q$. Then II defeats $\varphi_B$ by playing $(\corner{ \sigma  , \varphi_B^{\#}}, w,z)$, where $\sigma^{L[\varphi_B]}(v) = $the $v$-th $\varphi_B$-admissible ordinal. If $[f]^Q \notin A$, II can defeat $\varphi_C$ by a symmetrical argument. This is a contradiction. 


 \end{proof}

\begin{mydefinition}\label{def:S-P-Q-measure}
Let $Q$ be a finite level $\leq 2$ tree.
  We define 
\begin{displaymath}
    A \in \mu^{Q}
  \end{displaymath}  
iff there is $C \in \mu_{\mathbb{L}}$ such that 
 \begin{displaymath}
    [C]^{Q\uparrow} \subseteq A.
  \end{displaymath}
 \end{mydefinition}
 $\mu^{Q}$ is easily verified to be a countably complete filter concentrating on  $[\omega_1]^{Q \uparrow}$. In particular, when $\card(Q) = 1$, $\mu^Q$ is the principal measure concentrating on $\se{( \omega_1)_{(2,\emptyset)}}$. 
Noticing the facts that $\rep(Q)$ has order type $\omega_1+1$, and that $[f]^Q$ depends only on $\set{f(v)}{ \wocode{v}_{<^Q} \text{ is a limit ordinal}}$. 
Theorem~\ref{thm:strong_partition_omega_1} implies that 
\begin{displaymath}
\mu^{Q}\text{ is an }\admistwobold\text{-measure.}
\end{displaymath}
Let $j^{Q} = j^{\mu^Q}_{\admistwobold}$ be the restricted ultrapower map of $\mu^{Q}$ on $\admistwobold$.  Put $[f]_{\mu^Q} = [f]^{\mu^Q}_{\admistwobold}$ for $f \in \admistwobold$. \Los{}' theorem reads: for any first order formula $\varphi$, for any $x\in \mathbb{R}$,  for any $f_i  \in \admistwobold$,
with $\ran(f_i) \subseteq \admistwo{x}$, $1\leq i \leq n$,
\begin{displaymath}
j^Q(\admistwo{x}) \models \varphi([f_1]_{\mu^{Q}},\ldots,[f_n]_{\mu^{Q}})
\end{displaymath}
iff
\begin{displaymath}
\text{for $\mu^{Q}$-a.e.\ $\vec{\xi}$, } \admistwo{x} \models \varphi(f_1(\vec{\xi}), \ldots,f_n(\vec{\xi})).
\end{displaymath}

If $Q$ is a subtree of $Q'$, both finite,  then 
$\mu^{Q'}$  projects to $\mu^{Q}$ via the map that sends
$
     (\mathbb{\beta}_{q})_{(d,q)\in \dom(Q') }
 $to 
$     (\comp{d}{\beta}_{q})_{(d,q)\in \dom(Q) }
$.
Let
\begin{displaymath}
j^{Q,Q'}:\ult(\admistwobold, \mu^{Q}) \to\ult(\admistwobold, \mu^{Q'})
\end{displaymath}
be the induced factor map.
If $\vec{Q} = (Q_n)_{n<\omega}$ is a level $\leq 2$ tower, the associated $\admistwobold$-measure tower $(\mu^{Q_n})_{n<\omega}$ is easily seen close to $\admistwobold$.

The homogeneity property of the Martin-Solovay tree on a $\boldpi{2}$ set (cf. \cite{kechris_homo_tree_cabal}) translates to our context:
\begin{mytheorem}
  \label{thm:MS-tree-level-leq-2}
Assume $\boldDelta{2}$-determinacy. Let  $\vec{Q} = (Q_n)_{n<\omega}$ be an infinite level-2 tower. Let $Q_{\omega} = \cup_{n<\omega} Q_n$. The following are equivalent.
\begin{enumerate}
\item $Q_{\omega}$ is $\Pi^1_2$-wellfounded.
\item $<^{Q_{\omega}}$ is a wellordering.
\item There is $\vec{\beta} = (\beta_t)_{t \in \dom(Q_{\omega})}$ which respects $Q_{\omega}$. 
\item  $(\mu^{Q_n})_{n<\omega}$ is $\admistwobold$-countably complete.
\item The direct limit of $(j^{Q_m,Q_{n}})_{m<n<\omega}$ is wellfounded.
\end{enumerate}
\end{mytheorem}
\begin{proof}
1 $\Leftrightarrow$ 2: By Proposition~\ref{prop:level_2_tower_wellfounded_equivalence}.

2 $\Rightarrow$ 4: 
  Suppose  $<^{Q_{\omega}}$ is a wellordering. Let $(A_n)_{n<\omega}$  be such that $A_n \in \mu^{Q_n}\cap \admistwobold$. Let $x\in \mathbb{R}$ and $C\in L[x]$ be a club in $\omega_1$ such that  $[C]^{Q_n\uparrow} \subseteq A_n$ for all $n$.   Let $f : \dom( <^{Q_\omega}) \to C$ be given by
  \begin{displaymath}
    f(\vec{\alpha} \oplus_{Q_{\omega}} t ) = \text{ the } \wocode{\vec{\alpha} \oplus_{Q_{\omega}} t}_{<^{Q_\omega}}  \text{-th element of } C.
  \end{displaymath}
Then $f \in L[x,Q_\omega]$ and is order preserving. Let $\beta_n = [f\res \rep({Q_n})]^{Q_n}$. Then for all $n$, $(\beta_1,\ldots,\beta_n) \in A_n$.

4 $\Rightarrow$ 3: This follows from the fact that $\mu^{Q_n}$ concentrates on tuples that respect $Q_n$.

3 $\Rightarrow$ 1: 
If $x \in [\dom(Q_\omega)]$, then  $j^{Q_\omega{(x \res k)}, Q_\omega{(x \res l)}}(\beta_{x \res k})>\beta_{x \res l}$ for all $k<l<\omega$. This means the direct limit of $j^{Q_\omega{(x \res k)}, Q_\omega{(x \res l)}}$  is illfounded. Hence $Q_{\omega}(x)$ is not $\Pi^1_1$-wellfounded by Theorem~\ref{thm:martin_pi11_determinacy}. 

4 $\Leftrightarrow$ 5: By Proposition~\ref{prop:countable-complete-to-wellfounded}.
\end{proof}

\begin{mydefinition}
  \label{def:typical_level_2_trees}
 $Q^0$, $Q^1$, $Q^{20}$, $Q^{21}$ denote the following typical level $\leq 2$ trees of cardinalities at most 2:
\begin{itemize}
\item $\comp{1}{Q}^{0} =\emptyset$, $ \comp{1}{Q}^1 = \se{ (0)}$, $\dom(\comp{2}{Q}^0) = \dom(\comp{2}{Q}^1) = \se{\emptyset}$.
\item For $d \in \se{0,1}$, $\comp{1}{Q}^{2d}=\emptyset$, $\dom(\comp{2}{Q}^{2d}) = \se{\emptyset, ((0))}$, $\comp{2}{Q}^{2d}((0))$ is of degree $d$. 
\end{itemize}
\end{mydefinition}

$\mu^{Q^0}$ is a principle measure.
$\mu^{Q^1}$ is essentially $\mu_{\mathbb{L}}$. $\mu^{Q^{20}}$ and $\mu^{Q^{21}}$ are essentially refinements of the $\admistwobold$-club filter on $u_2$, the former concentrates on ordinals of $\admistwobold$-cofinality $\omega$, the latter of $\admistwobold$-cofinality $\omega_1$.

\subsection{The tree $S_3$}
\label{sec:weak-partition-LT3-measures}


A \emph{partial level $\leq 2$ tree} is a pair $(Q,(d,q,P))$ such that $Q $ is a finite level $\leq 2$ tree, and one of the following holds:
\begin{enumerate}
\item  $(d,q,P)=(0,-1,\emptyset)$, or
\item $d=1$,  $q \notin \comp{1}{Q}$, $\comp{1}{Q}\cup \se{q}$ is a level-1 tree, $P = \emptyset$, or
\item $d=2$,  $q \notin \dom(\comp{2}{Q}) $, $ \dom(\comp{2}{Q})\cup \se{q}$ is tree of  level-1 trees, $P$ is the completion of $\comp{2}{Q}(q^{-})$. (In particular, $\comp{2}{Q}(q^{-})$ must have degree 1.)
\end{enumerate}
The \emph{degree} of $(Q,(d,q,P))$ is $d$. 
We put $\dom(Q,(d,q,P)) = \dom(Q) \cup \se{(d,q)}$. 
The \emph{cardinality} of $(Q,(d,q,P))$ is $\card(Q,(d,q,P))= \card(Q)+1$. 
The \emph{uniform cofinality} of a partial level $\leq 2$ tree $(Q, (d,q,P))$ is
\begin{displaymath}
  \ucf(Q, (d,q,P)),
\end{displaymath}
defined as follows. 
\begin{enumerate}
\item $\ucf(Q,(d,q,P)) = (0,-1)$ if $d=0$;
\item $\ucf(Q,(d,q,P)) = (1, q^{-})$ if $d= 1$, $\lh(q) > 1$;
\item $\ucf(Q,(d,q,P)) = (2, (\emptyset, \emptyset, (0)))$ if $d= 1$, $\lh(q) = 1$;
\item $\ucf(Q,(d,q,P)) = (2, (q', P, \vec{p}))$ if $d = 2$, $\comp{2}{Q}[q'] = (P,\vec{p})$,  and 
$q'$ is the $<_{BK}$-least element of $\comp{2}{Q}\se{q,+}$, $q' \neq q^{-}$;
\item $\ucf(Q,(d,q,P)) = (2, (q^{-}, P, \vec{p}))$ if $d = 2$, $\comp{2}{Q}[q^{-}] = (P^{-},\vec{p})$,  and  $\comp{2}{Q}\se{q,+} = \se{q^{-}}$. 
\end{enumerate}
So $\ucf(Q,(d,q,P))$ is either $(0,-1)$ or a regular extended $Q$-description. The \emph{cofinality} of $(Q,{(d,q,P)})$ is
\begin{displaymath}
  \cf(Q, {(d,q,P)}) =
  \begin{cases}
    0 & \text{if } d = 0, \\
    1 & \text{if } d=1 \text{ and }q = \min(\prec^{\comp{1}{Q}\cup \se{q}}), \\
    2 & \text{otherwise.}
  \end{cases}
\end{displaymath}

A tuple  $\vec{\beta}= (\comp{e}{\beta}_t)_{(e,t) \in \dom(Q,(d,q,P))}$ \emph{respects} $(Q,(d,q,P))$ iff $\vec{\beta}\res \dom(Q)$ respects $Q$ and $\comp{d}{\beta}_q<\omega$ if $d=0$, 
$\vec{\beta}$ respects a completion of $(Q,(d,q,P))$ otherwise.
A partial level $\leq 2$ tree of degree 0 has no completion. 
A \emph{completion} of a partial level $\leq 2$ tree $(Q,(d,q,P))$ of degree $\geq 1$ is a level $\leq 2$ tree $Q^*$ such that $\dom(Q^*) = \dom(Q,(d,q,P))$, $\comp{2}{Q}^* \res \dom(\comp{2}{Q}) = \comp{2}{Q}$, and either $d=1$ or $d=2\wedge \comp{2}{Q}_{\tree}(t) = P$. 
  For a level $\leq 2$ tree $Q'$, $(Q,(d,q,P))$ is a \emph{partial subtree} of $Q'$ iff a completion of $(Q,(d,q,P))$ is a subtree of $Q'$.

A \emph{partial level $\leq 2$ tower of discontinuous type} is a nonempty 
finite sequence $(Q_i,(d_i,q_i,P_i))_{1 \leq i \leq k}$ such that $\card(Q_1)=1 $, each $(Q_i,(d_i,q_i,P_i))$ is a partial level $\leq 2$ tree, and each $Q_{i+1}$ is a completion of $(Q_i,(d_i,q_i,P_i))$. 
A \emph{partial level $\leq 2$ tower of continuous type} is $(Q_i,(d_i,q_i,P_i))_{1 \leq i < k} \concat (Q_{*})$ such that either $k=0 \wedge Q_{*}$ is the level $\leq 2$ tree of cardinality 1 or $(Q_i,(d_i,q_i,P_i))_{1 \leq i < k}$ is a partial level $\leq 2$  tower of discontinuous type $\wedge Q_{*}$ is a completion of $(Q_{k-1}, (d_{k-1},q_{k-1},P_{k-1}))$. 
For notational convenience, the information of a partial level $\leq 2$ tower is compressed into a potential partial level $\leq 2$ tower. 
A \emph{potential partial level $\leq 2$ tower} is $(Q_{*},\overrightarrow{(d,q,P)}) = (Q_{*}, (d_i, q_i, P_i)_{1 \leq i \leq \lh(\vec{q})})$ such that for some level $\leq 2$ tower $\vec{Q} = (Q_i)_{1 \leq i \leq k}$, either $Q_{*} = Q_k$ $\wedge$ $(\vec{Q}, \overrightarrow{(d,q,P)})$ is a partial level $\leq 2$ tower of discontinuous type or $(\vec{Q}, \overrightarrow{(d,q,P)}) \concat (Q_{*})$ is a partial level $\leq 2$ tower of continuous type. 

\begin{mydefinition}\label{def:level_3_tree}
  A \emph{level-3 tree of uniform cofinality}, or \emph{level-3 tree}, is a function
  \begin{displaymath}
R
\end{displaymath}
such that $\emptyset \notin \dom(R)$, $\dom(R)\cup \se{\emptyset}$ is tree of level-1 trees and for any $r \in \dom(R)$, $(R (r \res l))_{ 1 \leq l \leq \lh(r)}$ is a partial  level $\leq 2$ tower of discontinuous type.
If $R(r) = (Q_r, (d_r,q_r, P_r))$, we denote  $R_{\tree}(r) = Q_r$, $R_{\node}(r) = (d_r,q_r)$, 
 $R[r] = (Q_r,  (d_{r \res l}, q_{r \res l}, P_{r \res l})_{1 \leq l \leq  \lh(r)})$. $R[r]$ is a potential partial level $\leq 2$ tower of discontinuous type. If $Q$ is a completion of $R(r)$, put $R[r,Q] = (Q, (d_{r \res l}, q_{r \res l}, P_{r \res l})_{1 \leq l \leq  \lh(r)})$, which is a potential partial level $\leq 2$ tower of continuous type. For $r \in \dom(R)\cup \se{\emptyset}$, put $R\se{r} = \set{a \in \omega^{<\omega}}{r \concat (a) \in \dom(R)}$, which is a level-1 tree. 

  The \emph{cardinality} of $R$ is $\card(R) = \card(\dom(R))$.  $R$ is said to be \emph{regular} iff $((1)) \notin \dom(R)$.  In other words, when $R \neq \emptyset$, $((0))$ is the $<_{BK}$-maximum of $\dom(R)$. 
\end{mydefinition}

Suppose  $R$ is a level-3 tree. Let $\bardom(R) = \dom(R)\cup \set{r \concat (-1)}{r \in \dom(R) }$. For $r \in \dom(R)$, put $R\se{r,- } = \se{r^{-} \concat (-1)} \cup \set{r^{-}\concat (a)}{R_{\tree}(r^{-}\concat (a)) = R_{\tree}(r), a <_{BK} r (\lh(r)-1)}$, 
 $R\se{r,- } = \se{r^{-}} \cup \set{r^{-}\concat (a)}{R_{\tree}(r^{-}\concat (a)) = R_{\tree}(r), a >_{BK} r (\lh(r)-1)}$, 

If $\vec{\beta} = (\comp{d}{\beta}_q)_{(d,q) \in N}$ is a tuple indexed by $N$, $r \in \bardom(R)$, $\lh(r) = k$, either $k=1$ or $\dom(R(r^{-})) \subseteq N$, we put
\begin{displaymath}
  \vec{\beta} \oplus_R r = (r(0), \comp{d}{\beta}_{q_1}, r(1), \dots, \comp{d}{\beta}_{q_{k-1}}, r(k-1)), 
\end{displaymath}
where $(d_i,q_i) = R_{\node}(r \res i)$. 
The \emph{ordinal representation} of $R$ is the set
\begin{align*}
  \rep(R) =& \set{\vec{\beta} \oplus_R r}{r \in \dom(R), \vec{\beta} \text{ respects } R_{\tree}(r)} \\
 & \cup \set{\vec{\beta} \oplus_R r \concat(-1)}{r \in \dom(R), \vec{\beta} \text{ respects }R(r)}. 
\end{align*}
$\rep(R)$ is endowed with the $<_{BK}$ ordering
\begin{displaymath}
  <^R = <_{BK} \res \rep(R).
\end{displaymath}
$R$ is \emph{$\Pi^1_3$-wellfounded} iff
\begin{enumerate}
\item $\forall r \in \dom(R) \cup \se{\emptyset}~ R\se{r}$ is $\Pi^1_1$-wellfounded, and
\item 
 $\forall z\in [\dom(R)] ~ R(z) \DEF \cup_{n<\omega}(R_{\tree}(z \res n))_{1 \leq n<\omega}$ is not $\Pi^1_2$-wellfounded.
 \end{enumerate}

For level-3 trees $R$ and $R'$, $R$ is a \emph{subtree} of $R'$ iff $R$ is a subfunction of $R'$. 
A \emph{finite level-3 tower} is a sequence $(R_i)_{i \leq n}$ such that $n<\omega$, each $R_i$ is a regular level-$2$ tree, $\card(R_i) = i+1$ and $i<j \to R_i$ is a subtree of $R_{j}$. $\vec{R}$ is \emph{regular} iff each $R_i$ is regular. 
An \emph{infinite level-3 tower} is a sequence $\vec{R}=(R_n)_{n<\omega}$ such that for each $n$, $(R_i)_{i \leq n}$ is a finite level-3 tower. 
 $\Pi^1_3$-wellfoundedness of a level-$3$ tower is a $\boldpi{3}$ property in the real coding the tower. In particular, every finite level-3 tree is $\Pi^1_3$-wellfounded. 
Similarly to Proposition~\ref{prop:level_2_tower_wellfounded_equivalence}, we have
\begin{myproposition}
  \label{prop:level_3_tower_wellfounded_wellorder_equivalent}
Assume $\boldDelta{2}$-determinacy. Suppose $R$ is a level-3 tree. Then $R$ is $\Pi^1_3$-wellfounded iff $<^R$ is a wellordering.
\end{myproposition}

Associated to a $\boldpi{3}$ set $A$ we can assign a level-3 system  $(R_s)_{s \in \omega^{<\omega}}$ so that 
$x\in A$ iff the infinite level-3 tree $ R_x \DEF  \cup_{n<\omega}R_{x \res n}$ is $\Pi^1_3$-wellfounded. If $A$ is lightface $\Pi^1_3$, then $(R_s)_{s \in \omega^{<\omega}}$   can be picked effective.

Suppose  $F \in \admistwobold$ is a function on $\rep(R)$, $r \in \dom(R) $. Then $F_r$ is a function on $\omega_1^{R_{\tree}(r) \uparrow}$ that  sends $\vec{\beta}$ to $F(\vec{\beta} \oplus_R r)$.   $F$ represents a $\card(R)$-tuple of ordinals
\begin{displaymath}
  [F]^R = ([F]^R_r)_{r \in \dom(R)}
\end{displaymath}
where 
$[F]^R_r = [F_r]_{\mu^{R_{\tree}(r)}}$  for $r \in \dom(R)$. 
If $B \subseteq \bolddelta{3}$, put
\begin{displaymath}
  F \in B^{R\uparrow}
\end{displaymath}
iff $F \in \admistwobold$ and $F$ is an order-preserving continuous function from $\rep(R)$ to $B$ (with respect to $<^R$ and $<$). Let
\begin{displaymath}
  [B]^{R \uparrow} = \set{[F]^R}{ F \in B^{R \uparrow}}. 
\end{displaymath}
A tuple of ordinals $\vec{\gamma} = (\gamma_r)_{r \in \dom(R)}$ is said to \emph{respect} $R$ iff $\vec{\gamma} \in [\bolddelta{3}]^{R \uparrow}$. $\vec{\gamma}$ is said to \emph{weakly respect} $R$ iff for any $t,t'\in \dom(R)$, if   $t$ is a proper initial segment of $t'$, then $j^{R_{\tree}(t),R_{\tree}(t')}(\gamma_t) > \gamma_{t'}$.
By virtue of the order $<^R$,  if $\vec{\gamma}$ respects $R$, then $\vec{\gamma}$ weakly respects $R$ and whenever $R_{\tree}(t \concat (p)) = R_{\tree}(t \concat (q))$ and $p<q$, then $\gamma_{t \concat (p)}<\gamma_{t\concat (q)}$.

 The trees $S_3^{-}$ and $S_3$ are defined in \cite{kechris_homo_tree_cabal}.  They both project to the universal $\Pi^1_3$ set. In our language, they take the following form.
\begin{mydefinition}
  \label{def:S_3}
  Assume $\boldDelta{2}$-determinacy.
  \begin{enumerate}
  \item $S_3^{-}$ is the tree on $V_{\omega} \times \bolddelta{3}$ such that $(\emptyset,\emptyset) \in S_3^{-}$ and  
    \begin{displaymath}
      (\vec{R}, \vec{\alpha} )  = ( (R_i)_{ i\leq n}, (\alpha_i)_{i \leq n}  ) \in S_3^{-}
    \end{displaymath}
    iff $\vec{R}$ is a finite regular level-3 tower and letting $r_i \in \dom(R_{i+1}) \setminus \dom(R_i)$, $\beta_{r_i} = \alpha_{i+1}$, then  $(\beta_r)_{r \in \dom(R_n)}$  respects $R_n$.  
  \item $S_3$ is the tree on $V_{\omega} \times \bolddelta{3}$ such that $(\emptyset,\emptyset) \in S_3$ and 
    \begin{displaymath}
      (\vec{R}, \vec{\alpha} )  = ( (R_i)_{ i\leq n}, (\alpha_i)_{i \leq n}  ) \in S_3
    \end{displaymath}
    iff $\vec{R}$ is a finite regular level-3 tower and letting $r_i \in \dom(R_{i+1}) \setminus \dom(R_i)$, $\beta_{r_i} = \alpha_{i+1}$, then  $(\beta_r)_{r \in \dom(R_n)}$  weakly respects $R_n$.    
\end{enumerate}
\end{mydefinition}
 By Theorem~\ref{thm:MS-tree-level-leq-2},
 \begin{displaymath}
   p[S_3^{-}]  = p[S_3]  =  \set{\vec{R}}{\vec{R} \text{ is a $\Pi^1_3$-wellfounded level-3 tower} }.
 \end{displaymath}
The (non-regular) scale associated to $S_3$ is $\Pi^1_3$. For $\xi < \bolddelta{3}$, put $(\vec{R},\vec{\alpha}) \in S_3\res \xi$ iff $(\vec{R}, \vec{\alpha}) \in S_3$ and  $(\vec{R},\vec{\alpha}) \neq (\emptyset,\emptyset) \to \alpha_0 < \xi$.

Suppose $E$ is a club in $\omega_1$. For a partial level $\leq 2$ tree $(Q,(d,q,P))$, put $\vec{\alpha} = (\comp{e}{\alpha}_t)_{(e,t) \in \dom(Q, (d,q,P)) } \in [E]^{(Q, (d,q,P)) \uparrow}$ iff $\vec{\alpha}$ respects $(Q,(d,q,P))$, $(\comp{e}{\alpha}_t)_{(e,t) \in \dom(Q)} \in [E]^{Q \uparrow}$,  and $d=1 \to \comp{1}{\alpha}_q \in E$, $d=2 \to \comp{2}{\alpha}_q \in j^P(E)$.
For a level-3 tree $R$, put 
 \begin{align*}
   \rep(R) \res E= &  \set{ \vec{\beta} \oplus_R r}{ r \in \dom(R), \vec{\beta} \in [E]^{R_{\tree}(r) \uparrow}} \\
   & \cup \set{ \vec{\beta} \oplus_R r \concat (-1)}{r \in \dom(R),  \vec{\beta} \in [E]^{R(r)\uparrow}}.
 \end{align*}
By Lemma~\ref{lem:Q_respecting},
 $\rep(R) \res E$ is a closed subset of $\rep(R)$ (in the order topology of $<^R$). A useful consequence is that the order preserving bijection
 \begin{displaymath}
   \theta_R^E : \rep(R) \res E \to \rep(R)
 \end{displaymath}
is the identity on $\rep(R) \res E'$ for a club $E' \subseteq E$.

\section{The lightface level-3 sharp}
\label{sec:level-3-sharp}

This section defines a real $0^{3\#}$ which is many-one equivalent to $M_2^{\#}$, under boldface $\boldpi{3}$-determinacy. The assumption of $\boldpi{3}$-determinacy is very likely not optimal.  

\subsection{Level-3 boundedness}
\label{sec:level-2-admissible}

Recall in Corollary~\ref{coro:Kunen_Martin_Sigma13} that the rank of a $\Sigma^1_3(\lessthanshort{u_{\omega}},x)$ wellfounded relation is bounded by $\kappa_3^x$. 
We would like to strengthen this fact by allowing a suitable code for an arbitrary ordinal in $\bolddelta{3}$. 
The strengthening is based on an inner model theoretic characterization of $u_{\omega}$ in $L[T_3,x]$. 
%
%
We say that 
\begin{displaymath}
  \delta \text{ is an $L$-Woodin cardinal}
\end{displaymath}
iff  $L(V_{\delta})\models \delta$ is Woodin.

\begin{mytheorem}[Woodin, {\cite[Theorem 5.22]{sarg_pwo_2013}}]
  \label{thm:u_omega_L_woodin}
Assume $\boldpi{3}$-determinacy. Let $\kappa = u_{\omega}$.  For $x \in \mathbb{R}$, $M_{2, \infty}^{-}(x) \models \kappa$ is the least $L$-Woodin cardinal.
\end{mytheorem}

\begin{mycorollary}[Level-3 boundedness]
  \label{coro:Sigma13-boundedness}
Assume $\boldpi{3}$-determinacy. Suppose $x \in \mathbb{R}$,  $\mathcal{N} \in \mathcal{F}_{2,x}$, $\eta$ is a cardinal and strong cutpoint of $\mathcal{N}$, $\xi = \pi_{\mathcal{N},\infty}(\eta)$. 
Suppose  $g$ is $\coll(\omega,\eta)$-generic over $\mathcal{N}$, $r \in \mathbb{R}\cap \mathcal{N}[g]$.
Let $\lambda$ be the least $L$-Woodin cardinal in $M_{2,\infty}^{-}(x)$ above $\xi$.
Suppose $G$ is a $\Pi^1_3(r, \lessthanshort{u_{\omega}})$ set equipped with a regular $\Pi^1_3(r , \lessthanshort{u_{\omega}})$ norm $\varphi$. Suppose $A$ is a $\Sigma^1_3(r, \lessthanshort{u_{\omega}})$ subset of $G$. Then 
\begin{displaymath}
   \sup\set{ \varphi(y)}{y \in A} < (\lambda^{+})^{M_{2,\infty}(x)}.
\end{displaymath}
\end{mycorollary}
\begin{proof}
Put $x=0$ for simplicity. Put
\begin{displaymath}
  \mathcal{G}_{2}^{\mathcal{N},\eta} = \{ \mathcal{P} \in \mathcal{F}_{2} : \mathcal{P} \text{ is a nondropping iterate of } \mathcal{N} \text{ above } \eta \}.
\end{displaymath}
$\mathcal{G}_{2}^{\mathcal{N},\eta} $ is a subsystem of $\mathcal{F}_{2}$. Let $M_{2,\infty}^{\mathcal{N},\eta,\#}$ be the direct limit of $\mathcal{G}_{2}^{\mathcal{N},\eta}$. The inclusion map of  direct systems induces an embedding between direct limits
\begin{displaymath}
  \pi^{\mathcal{N},\eta}_{x} : M_{2,\infty}^{\mathcal{N},\eta,\#} \to M_{2,\infty}^{\#}.
\end{displaymath}
Let $r_g \in \mathbb{R}$ be the real coding $(g, \mathcal{N}|\eta)$. 
Every mouse $\mathcal{P} \in \mathcal{G}_{2}^{\mathcal{N},\eta}$ corresponds to an $r_g$-mouse $\mathcal{P}[g] \in \mathcal{F}_{2,r_g}$ (converted into an $r_g$-mouse in the obvious way, cf. \cite{schindler-steel-self-iterability}). So in the direct limit,
\begin{displaymath}
  M_{2,\infty}^{\mathcal{N},\eta,\#} [g] = M_{2,\infty}^{\#}(r_g).
\end{displaymath}
By Corollary~\ref{coro:Kunen_Martin_Sigma13}, 
\begin{displaymath}
   \sup\set{ \varphi(y)}{y \in A} < \kappa_3^{r_g},
\end{displaymath}
which in turn is smaller than the successor of $u_{\omega}$ in $M_{2,\infty}^{\#}(r_g)$, as $\se{T_2,r_g} \in M_{2,\infty}^{\#}(r_g)$. By Theorem~\ref{thm:u_omega_L_woodin}, $u_{\omega}$ is the least $L$-Woodin cardinal of $M_{2,\infty}^{\#}(r_g)$, hence the least $L$-Woodin cardinal of 
$M_{2,\infty}^{\mathcal{N},\eta,\#}$ above $\eta$. By elementarity, 
$\pi_x^{\mathcal{N},\eta} (u_{\omega}) = \lambda$. So $\pi_x^{\mathcal{N},\eta}(\kappa_3^{r_g}) < (\lambda^{+})^{M_{2,\infty}}$. This finishes the proof. 
\end{proof}

\subsection{Representations of ordinals in \texorpdfstring{$\bolddelta{3}$}{}}
\label{sec:Pi13_norm_vs_wellordering_on_u_omega}

We introduce a coding system for ordinals in $\bolddelta{3}$ which is the higher level analog of $\WO$. The coding system is guided by Theorem~\ref{thm:Delta13_coding}. Identifying $u_{\omega}$ with $(V_{\omega} \cup u_{\omega})^{<\omega}$, we shall assume $X$ is a $\Delta^1_3$ subset of $\mathbb{R} \times (V_{\omega} \cup u_{\omega})^{<\omega}$ so that the map $v \mapsto X_v$ is a surjection  from $\mathbb{R}$ onto $\power((V_{\omega}\cup u_{\omega})^{<\omega})$.

For a finite level-3 tree $R$ and a tuple $\vec{\beta} \oplus_R t \in \rep(R)$, put
\begin{displaymath}
v\in \LO^R_{\vec{\beta} \oplus_R t}
\end{displaymath}
iff  for each $\vec{\gamma} \oplus_R s  \leq^R  \vec{\beta} \oplus_R t$,
\begin{displaymath}
  (X_v)_{\vec{\gamma} \oplus_R s} \DEF \set{(\xi,\eta)}{(v,\vec{\gamma}\oplus_R s, \xi,\eta) \in X_v}
\end{displaymath}
 is a linear ordering on $u_{\omega}$. 
 Put 
 \begin{displaymath}
   v \in \LO^R
 \end{displaymath}
iff $v \in \LO^R_{\vec{\beta}\oplus_R t}$ for all $\vec{\beta}\oplus_R t\in \rep(R)$. 
The relations ``$ v\in \LO^R_{\vec{\beta} \oplus_R t}$'' and ``$v\in \LO^R$'' are $\Delta^1_3$.
Put
 \begin{displaymath}
   v \in \WO^{R\uparrow}_{\vec{\beta} \oplus_R t}
 \end{displaymath}
iff  for each  $\vec{\gamma} \oplus_R s  \leq ^R \vec{\beta} \oplus_R t$, $(X_v)_{\vec{\gamma} \oplus_R s}$ is a wellordering on $u_{\omega}$, and the map $\vec{\gamma} \oplus _Rs \mapsto  \ot((X_v)_{\vec{\gamma}\oplus_R s})$ is continuous, order preserving for  $\vec{\gamma} \oplus_R s \leq ^R \vec{\beta} \oplus_R t$. Put  
 \begin{displaymath}
   v \in \WO^{R\uparrow}
 \end{displaymath}
iff $v \in \WO^{R\uparrow}_{\vec{\beta} \oplus_R t}$ for all $\vec{\beta} \oplus_R t\in \rep(R)$. 
The relations
``$v\in \WO^{R\uparrow}_{\vec{\beta} \oplus_R t}$"
and 
``$R \text{ is a finite level-3 tree } \wedge v\in \WO^{R\uparrow}$''
are $\Pi^1_3$. If $(X_v)_{\vec{\beta} \oplus_R t}$ is a wellordering on $u_{\omega}$, its order type is denoted by $\wocode{v}_{\vec{\beta} \oplus_R t}$. A member $v \in \WO^{R\uparrow}$ codes a tuple of ordinals $[v]^R$ that respects $R$:
\begin{displaymath}
  [v]^R = 
  [\vec{\beta} \oplus_R t \mapsto \wocode{v}_{\vec{\beta}\oplus_R t}]^R.
\end{displaymath}

Clearly, if  $v \in \WO^{R \uparrow}$, then $[v]^R\in \admistwo{v,R}$ and  is $\Delta_1$-definable in $\admistwo{v,R}$ from $\se{T_2,v,R}$. Put $[v]^R = ([v]^R_t)_{t \in \dom(R)}$. So $[v]^R_t = [\vec{\beta} \mapsto \wocode{v}^R_{\vec{\beta}\oplus_R t}]_{\mu^{R_{\tree}(t)}}$.

Observe the simple fact that for any finite level-1 tree $W$, for any $\vec{\alpha}=(\alpha_w)_{w \in W} $ respecting $W$, there is a $\Pi^1_1$-wellfounded level-1 tree $W'$ extending $W$ such that  $\alpha_w=\wocode{(w)}_{<^{W'}}$ for any $w \in W$. 
Intuitively, $W'$ ``represents'' $\vec{\alpha}$ in the sense that $\vec{\alpha}$ extends to a tuple $\vec{\alpha}'$ respecting $W'$ and if $\vec{\beta}$ respects $W'$, then $\forall w \in W~ \alpha_w \leq \beta_w$.
 It is implicitly used in proving that $0^{\#}$ is the unique wellfounded remarkable EM blueprint. 
Likewise, its higher level analog will be an ingredient in the level-3 EM blueprint formulation of $0^{3\#}$.

We also need to code ordinals in $\bolddelta{3}$ by direct limits of iterations of $\Pi^1_3$-iterable mice. Suppose $x \in \mathbb{R}$ and $z$ codes a $\Pi^1_3$-iterable $x$-mouse $\mathcal{P}_z$. Then
\begin{displaymath}
  \pi_{\mathcal{P}_z, \infty} : \mathcal{P}_z \to (\mathcal{P}_z)_{\infty}
\end{displaymath}
is the direct limit map of all the nondropping iterates of $\mathcal{P}_z$. $o((\mathcal{P}_z)_{\infty})$ is the length of a $\Delta^1_3(z)$-prewellordering, namely the one induced by iterations of $\mathcal{P}_z$. By Corollary~\ref{coro:Delta13_pwo_computable_in_LT2}, $\pi_{\mathcal{P}_z,\infty}$ and $(\mathcal{P}_z)_{\infty}$ are both in $\admistwo{M_1^{\#}(z)}$ and $\Delta_1$-definable over $\admistwo{M_1^{\#}(z)}$ from $\se{T_2, M_1^{\#}(z)}$. 

\subsection{Putative level-3 indiscernibles}
\label{sec:level_3_indiscernibles}

The higher level analog of the type of $L$ with $n$ indiscernibles is the type of $M_{2,\infty}^{-}$ realized by an appropriate $[F]^R$, where $F \in (\bolddelta{3})^{R \uparrow} $. Such functions $F$ are coded by subsets of $u_{\omega}$ in $\admistwobold$. The coding system is provided by Theorem~\ref{thm:Delta13_coding}.

 $\mathcal{L}=\se{\underline{\in}}$ is the language of set theory. 
For a level-3 tree $R$, $\mathcal{L}^R$ is the expansion of $\mathcal{L}$ which   consists of  additional
 constant symbols $\underline{c_r}$ for each $r \in \dom(R) $.
For a level-3 tree $R$ and a tuple of ordinals $\vec{\gamma} = (\gamma_r)_{r\in \dom(R)}$, the $\mathcal{L}$-structure $M_{2,\infty}^{-}$ expands to the $\mathcal{L}^R$-structure
\begin{displaymath}
  (M_{2,\infty}^{-} ; \vec{\gamma})
\end{displaymath}
 whose constant $\underline{c_r}$ is interpreted as $\gamma_r$.

 \begin{mydefinition}
   \label{def:firm}
 $C \subseteq \bolddelta{3}$ is said to be \emph{firm} iff every member of $C$ is additively closed, the set  $\set{\xi}{\xi = \ot(C \cap \xi)}$ has order type $\bolddelta{3}$  and $C \cap \xi \in \admistwobold$ for all $\xi < \bolddelta{3}$.
 \end{mydefinition}

\begin{mydefinition}
  \label{def:level-3_indiscernibles}
 $C\subseteq \bolddelta{3}$ is called a set of \emph{potential level-3 indiscernibles for $M_{2,\infty}^{-}$} iff  for any level-3 tree $R$, for any $F ,G \in C^{R \uparrow}\cap \admistwobold$, 
 \begin{displaymath}
   (M_{2,\infty}^{-}; [F]^R) \equiv    (M_{2,\infty}^{-}; [G]^R).
 \end{displaymath}
\end{mydefinition}
A firm set of potential level-3 indiscernibles for $M_{2,\infty}^{-}$ is the higher level analog of a set of order indiscernibles for $L$. Note that 
the successor elements of $C$ don't really play a part in computing $[F]^R = ([F_r]_{\mu^{R_{\tree}(r)}})_{r \in \dom(R)}$, as the relevant ultrapowers $\mu^{R_{\tree}(r)}$ concentrate on tuples of limit ordinals,  hence the prefix ``potential''. 

\begin{mylemma}
  \label{lem:the-measure-W3R}
Assume $\boldpi{3}$-determinacy. Then there is a firm set of potential level-3 indiscernibles for $M_{2,\infty}^{-}$.
\end{mylemma}
\begin{proof}  
Suppose $R$ is a finite level-3 tree. 
Let $\varphi$ be an $\mathcal{L}^R$-sentence. Consider the  game $G^{R;\varphi}$ where I produces reals $v,x,c$ and a natural number $p$, II produces reals $v',x',c'$ and a natural number $p'$. The payoff is decided according to the following priority list:
  \begin{enumerate}
  \item  I and II must take turns to ensure that $v \in \WO^{R \uparrow}$ and $v' \in \WO^{R \uparrow}$. If one of them fails to do so, and $w\in \rep(R)$ is $<^R$-least for which $v \notin \WO^{R \uparrow}_w \vee v' \notin \WO^{R \uparrow}_w$, then I loses iff $v \notin \WO^{R \uparrow}_w$, and II loses iff $v \in \WO^{R \uparrow}_w$.
  \item If 1 is satisfied, put $\vec{\gamma} = (\gamma_r)_{r \in \dom(R)}$, where $\gamma_r = \max([v]^R_r, [v']^R_r)$.
I must ensure
    \begin{enumerate}
    \item $x$ codes a 2-small premouse $\mathcal{P}_x$ which satisfies ``I am closed under the $M_1^{\#}$-operator'';
    \item $c$ codes a strictly increasing, cofinal-in-$o(\mathcal{P}_x)$ sequence of ordinals $(c_n)_{n < \omega}$ relative to $x$ such that each $c_n$ is a cardinal cutpoint of $\mathcal{P}_x$;
     \item  $\mathcal{P}_x | c_1$ is a $\Pi^1_3$-iterable mouse; 
    \item $p$ codes a tuple of ordinals $\vec{\alpha} = (\alpha_r)_{r \in \dom(R)}$ in $\mathcal{P}_x|c_0$ relative to $x$;
    \item For each $r \in \dom(R)$, $\pi_{\mathcal{P}_x|c_0, \infty}(\alpha_r) =\gamma_r$;
    \item $(\mathcal{P}_x; \vec{\alpha}) \models \varphi$.
    \end{enumerate}
 Otherwise he loses.
  \item If 1-2 are satisfied, II must ensure 2(a)-(f) with $(x,c,(c_n)_{n<\omega}, p,\vec{\alpha},\varphi)$ replaced by $(x',c', (c'_n)_{n<\omega}, p',\vec{\alpha'}, \neg \varphi)$,
 otherwise he loses.
  \item If 1-3 are satisfied, I and II must take turns to ensure for all $2 \leq n <\omega$,
    \begin{enumerate}
    \item $\mathcal{P}_x|c_n$ is a $\Pi^1_3$-iterable mouse and $\mathcal{P}_{x'}| c'_{n-1} <_{DJ} \mathcal{P}_x| c_n$;
    \item $\mathcal{P}_{x'}|c_n'$ is a $\Pi^1_3$-iterable mouse and $\mathcal{P}_{x}| c_{n} <_{DJ} \mathcal{P}_{x'}| c_n'$.
    \end{enumerate}
 If one of them fails to do so, and $n$ is least for which (a) or (b) fails at $n$,  then I loses iff  (a) fails at $n$, and II loses iff (a) holds at $n$.
  \item It is impossible that both players obey all the rules, due to a successful comparison between $\mathcal{P}_x$ and $\mathcal{P}_{x'}$. The definition of $G^{R;\varphi}$ is finished. 
  \end{enumerate}
 The payoff of $G^{R;\varphi}$ has complexity $(\llbracket \emptyset \rrbracket_R+\omega) \textnormal{-}\Pi^1_3$ for both players. The nontrivial part about the complexity is that 2(e) is $\Delta^1_3$, shown as follows. According to rules 2(a)-(c),  $\mathcal{P}_x|c_1$ is $\Pi^1_3$-iterable and closed under the (genuine) $M_1^{\#}$-operator, $c_0<c_1$,  and therefore 
$M_1^{\#}(\mathcal{P}_x | c_0)$ is canonically coded in $x$. 
 $\pi_{\mathcal{P}_x|c_0, \infty}(\alpha_s)$ is the length of a $\Delta^1_3(\mathcal{P}_x | c_0)$ prewellordering, induced by iterations. By Corollary~\ref{coro:Delta13_pwo_computable_in_LT2},  $\pi_{\mathcal{P}_x|c_0, \infty}(\alpha_s)$ is $\Delta_1$-definable over $\admistwo{x}$ from $\se{T_2,x}$. $\vec{\gamma}$ is clearly $\Delta_1$-definable over $\admistwo{v}$ from $\se{T_2,v}$.  So 2(e) is expressed into a $\Delta_1$ statement over $L_{\kappa_3^{v,x}}[T_2,v,x]$ from $\se{T_2, v,x,c}$, or equivalently,  $\Delta^1_3(v,x,c)$ by Theorem~\ref{thm:BK-KM}.

 Hence $G^{R;\varphi}$ is determined. Suppose for definiteness II has a winning strategy $\sigma$ in $G^{R;\varphi}$. Let $C$ be the set of $L$-Woodin cardinal cutpoints of $M_{2,\infty}^{-}(\sigma)$ and their limits. We show that
 \begin{displaymath}
   \forall  F \in C^{R \uparrow} ~ (M_{2,\infty}^{-}; [F]^R) \models \neg \varphi
 \end{displaymath}
Suppose towards a contradiction that $F \in C^{R \uparrow}$ but $(M_{2,\infty}^{-}; [F]^R) \models  \varphi$. As $\bolddelta{3}$ is inaccessible in $M_{2, \infty}^{\#}$, there is a club $D \in M_{2,\infty}^{\#}$ in $\bolddelta{3}$ so that $M_{2,\infty}^{-} | \lambda \elem M_{2,\infty}^{-}$ for any $\lambda \in D$. There is thus a continuous, order preserving $G : \omega+1 \to C \setminus \sup \ran( F)$ for which $(M_{2,\infty}^{-}| G(\omega); [F]^R) \models \varphi$. Pick $\mathcal{P} \in \mathcal{F}_{2}$ and ordinals $(c_n)_{n<\omega}$, $(\alpha_r)_{r \in \dom(R)}$ in $\mathcal{P}$ such that $\pi_{\mathcal{P},\infty}(c_n) = G(n)$ for any $n<\omega$ and $\pi_{\mathcal{P},\infty} (\alpha_r) = [F]^R_r$ for any $r  \in \dom(R)$. 
Thus, $(\mathcal{P} | \sup_{n<\omega} c_n; \vec{\alpha} )\models \varphi$. 
Let Player I play $(v,x,c,p)$, where $v \in \WO^{R \uparrow}$, $\wocode{v}^R_w = F(w)$ for any $w \in \rep(R)$, $x$ codes $\mathcal{P}|\sup_{n<\omega}c_n$, $c$ codes $(c_n)_{n<\omega}$ relative to $x$, $p$ codes $(\alpha_r)_{r \in \dom(R)}$. The response according to $\sigma$ is denoted by $(v',x',c',p') = (v,x,c,p)* \sigma$. We shall derive a contraction by showing neither player breaks the rules, using $\boldsigma{3}$-boundedness.

As $\sigma$ is  a winning strategy, Player II is not the first person to break the rules. So
$v \in \WO^{R \uparrow}$ implies $v' \in \WO^{R \uparrow}$. 
For each $w \in \rep(R)$ which is  either the $<^R$-minimum or a $<^R$-successor,  
  if $\mathcal{N} \in \mathcal{F}_{2,\sigma}$, $\eta \in \mathcal{N}$,  $\pi_{\mathcal{N},\infty}(\eta) = F(w)$,  $g$ is $\coll(\omega, \eta)$-generic over $\mathcal{N}$, $r_g\in \mathbb{R}$ being the real coding $(g, \mathcal{N}|\eta)$, then $(v',x',c',p')$ belongs to the set
 \begin{displaymath}
   A_w  =  \set{(\bar{v} ,\bar{x},\bar{c},\bar{p})*\sigma}{\bar{v} \in \WO^{R\uparrow}_w \res \xi}
 \end{displaymath}
which is $\Sigma^1_3(M_1^{\#}(r_g), \lessthanshort{u_{\omega}})$ by Corollary~\ref{coro:Delta13_pwo_computable_in_LT2} and Theorem~\ref{thm:BK-KM}. 
 Since $\sigma$ is a winning strategy, $A_w$  is a subset of 
\begin{displaymath}
B_w =      \set{(\bar{v}',\bar{x}',\bar{c}',\bar{p}')}{\bar{v}' \in \WO^{R\uparrow}_w}
\end{displaymath}
$B_w$ is a $\Pi^1_3(\lessthanshort{u_{\omega}})$ set, equipped with the $\Pi^1_3(\lessthanshort{u_{\omega}})$ prewellordering $(\bar{v}',\bar{x}',\bar{c}',\bar{p}') \mapsto \wocode{\bar{v}'}^R_w$.
By Corollary~\ref{coro:Sigma13-boundedness},  $\wocode{{v}'}^R_w < \min(C \setminus (F(w)+1))$. By continuity, if $w$ has $<^R$-limit order type, then $\wocode{{v}'}^R_w  \leq \wocode{{v}}^R_w $.
Consequently, for $r \in \dom(R)$, $[v']^R_r \leq   [v]^R_r$, so if $\vec{\gamma}$ is defined from $v,v'$ as in Rule 2, then $\gamma_r = [v]^R_r$.

By our choice of $F$ and $G$, Rule 2 is satisfied. 
  Let $\mathcal{P}_x, (c_n)_{n<\omega}, \vec{\alpha}, \mathcal{P}_{x'},(c'_n)_{n<\omega}, \vec{\alpha}'$ be defined as in Rules 2 and 3.
For each $1 \leq n<\omega$, using the $\Pi^1_3$-prewellordering on codes of $\Pi^1_3$-iterable mice, a similar boundedness argument shows that $\wocode{\mathcal{P}_{x'} | c_n'}_{<_{DJ}} <  \min (C \setminus (G(n)+1))$, and hence 
$\mathcal{P}_{x'} | c_{n}' <_{DJ} \mathcal{P}_x | c_{n+1}$. So Rule 4 is satisfied. This is impossible. 
\end{proof}

\begin{mydefinition} \label{def:x_3_sharp}
Assume $\boldpi{3}$-determinacy.  Let $C$ be a firm set of potential level-3 indiscernibles for $M^{-}_{2,\infty}$. Then
\begin{displaymath}
0^{3\#}
\end{displaymath}
is a map sending a finite level-3 tree $R$ to the complete consistent $\mathcal{L}^R$-theory $0^{3\#}(R)$, where $\gcode{\varphi} \in 0^{3\#}(R)$ iff $\varphi$ is an $\mathcal{L}^R$-formula and for all $\vec{\gamma} \in [C]^{R\uparrow}$, 
  \begin{displaymath}
    (M_{2,\infty}^{-}; \vec{\gamma}) \models \varphi.
  \end{displaymath}
\end{mydefinition}

$0^{3\#}$ is the higher level analog of $0^{\#}$.
Each individual $0^{3\#}(R)$ is the higher level analog of the $n$-type that is realized in $L$ by $n$ indiscernibles.  
%

The proof of Lemma~\ref{lem:the-measure-W3R} shows  

\begin{mylemma}
  \label{lem:3sharp_from_M2sharp}
 Assume $\boldpi{3}$-determinacy. For a finite level-3 tree $R$, $0^{3\#}(R)$ is a  $\game(\llbracket \emptyset \rrbracket_R+\omega)\textnormal{-}\Pi^1_3)$ real.
\end{mylemma}

\subsection{The equivalence of $x^{3\#}$ and $M_2^{\#}(x)$}
\label{sec:reduction-3sharp_to_M2sharp}

For the other direction of the reduction, we want to  compute $\game(\lessthanshort{u_{\omega}}\textnormal{-} \Pi^1_3)$ truth using $0^{3\#}$ as an oracle. 

\begin{mylemma}
  \label{lem:M2sharp_from_3sharp_level_by_level}
  Assume $\boldpi{3}$-determinacy. For  a finite level-3 tree ${R}$, the universal $\game( \llbracket \emptyset \rrbracket_R\textnormal{-} \Pi^1_3)$ real is many-one reducible to $0^{3\#}(R)$, uniformly in $ R$.
\end{mylemma}
\begin{proof}
Let $B \subseteq \llbracket \emptyset \rrbracket_R\times \mathbb{R}$ be $\Pi^1_3$. Let $\theta$ be a $\Sigma_1$ formula such that
\begin{displaymath}
  (\xi, x) \in B \eqiv \admistwo{x} \models \theta (\xi, x).
\end{displaymath} 
$G$ is the game with output $\Diff B$. 
We need to decide the winner of $G$ from $0^{3\#}(R)$. $B$ is equipped with the $\Pi^1_3$-norm
\begin{displaymath}
  \psi (\xi, x) = \text{the least } \alpha<\kappa_3^x \text{ such that } L_{\alpha}[T_2, x] \models \theta(\xi,x).
\end{displaymath}
If $E\in \mu_{\mathbb{L}}$ is a club, let $\rho^E: \llbracket \emptyset \rrbracket_R  \to \rep(R)\res E$ be the order preserving bijection.  
For $\vec{\gamma}$ respecting $R$, let  $\theta^{I}(\vec{\gamma})$ be the following formula:
\begin{quote}
  There exist $H \in (\bolddelta{3})^{R \uparrow}$ and a strategy $\tau$ for Player I such that $[H]^R = \vec{\gamma}$ and for any club $E \in \mu_{\mathbb{L}}$, if $x$ is an infinite run according to $\tau$, then 
  for any even $\alpha <\llbracket \emptyset \rrbracket_R$, $\forall \beta < \alpha ((\beta, x) \in B \wedge \psi(\beta, x) < H(\rho^E(\beta+1)))$ implies $(\alpha,x) \in B \wedge \psi(\alpha, x) < H(\rho^E(\alpha+1))$, and there is $\alpha < \llbracket \emptyset \rrbracket_R$ such that $(\alpha, x) \notin B$.
\end{quote}
Let  $\theta^{II}(\vec{\gamma})$ be the following formula:
\begin{quote}
  There exist $K \in (\bolddelta{3})^{R \uparrow}$ and a strategy $\sigma$ for Player II such that $[K]^R = \vec{\gamma}$ and for any club $E \in \mu_{\mathbb{L}}$, if $x$ is an infinite run according to $\sigma$, then for any odd $\alpha <\llbracket \emptyset \rrbracket_R$, $\forall \beta < \alpha((\beta, x) \in B  \wedge \psi(\beta, x) < K(\rho^E(\beta+1)))$ implies $(\alpha, x) \in B \wedge \psi(\alpha, x) < K(\rho^E(\alpha+1))$.
\end{quote}

 Let $C$ be  a firm set of level-3 indiscernibles for $M_{2, \infty}^{-}$. Suppose firstly Player I has a winning strategy $\tau$ in $G$. Let $D$ be the subset of $C$ consisting of  $L$-Woodin cardinals in $M_{2,\infty}(\sigma)$ and their limits. 
By Corollary~\ref{coro:Sigma13-boundedness}, if  $x$ is  a consistent run according to $\sigma$, then $(0,x) \in B\wedge \psi(0,x) < \min(D)$,  for any odd $\alpha < \llbracket \emptyset \rrbracket_R$,  $(\alpha, x) \in B$ implies $(\alpha+1,x) \in B \wedge \psi(\alpha+1, x) < \min(D \setminus (\psi(\alpha,x)+1))$, and there is $\alpha <\llbracket \emptyset \rrbracket_R$ such that $(\alpha, x) \notin B$. Let $H \in D^{R \uparrow}$. Then $(H, \tau)$ witnesses $\theta^I([H]^R)$. 
Let $\mathcal{P} \in \mathcal{F}_{2}$ and $\vec{\eta} \in \mathcal{P}$ such that $\pi_{\mathcal{P}, \infty}(\vec{\eta}) = [H]^R$.
Let $\xi_{\vec{\eta}}$  be the least successor cardinal cutpoint of $\mathcal{P}$ above $\max(\vec{\eta})$ and let $g$ be $\coll(\omega, \xi)$-generic over $\mathcal{P}$. Let $r_{g,\vec{\eta}}$ be the real coding $(g,\vec{\eta})$. Then $\theta^I([H]^R)$ is equivalent to a  $\Sigma^1_4(r_{g, \vec{\eta}})$ statement $\bar{\theta}^I(r_{g, \vec{\eta}})$, hence true in $\mathcal{P}[g]$. Hence,
\begin{displaymath}
  \mathcal{P}^{\coll(\omega, \xi_{\vec{\eta}})} \models \bar{\theta}^I(\dot{r}_{g, \vec{\eta}})
\end{displaymath}
By elementarity,
\begin{displaymath}
  (M_{2, \infty}^{-})^{\coll(\omega, \xi_{\vec{\gamma}})} \models \bar{\theta}^I (\dot{r}_{g, [H]^R}).
\end{displaymath}
By Lemma~\ref{lem:the-measure-W3R}, for any $\vec{\gamma} \in [C]^{R \uparrow}$, 
\begin{displaymath}
  (M_{2, \infty}^{-})^{\coll(\omega, \xi_{\vec{\gamma}})} \models \bar{\theta}^I (\dot{r}_{g,\vec{\gamma}}).
\end{displaymath}
By a symmetrical argument, if Player II has a winning strategy in $G$, then for any $\vec{\gamma} \in [C]^{R \uparrow}$, 
\begin{displaymath}
  (M_{2, \infty}^{-})^{\coll(\omega, \xi_{\vec{\gamma}})} \models \bar{\theta}^{II} (\dot{r}_{g,\vec{\gamma}}).
\end{displaymath}

Finally, there does not exist $\vec{\gamma}$ such that 
\begin{displaymath}
    (M_{2, \infty}^{-})^{\coll(\omega, \xi_{\vec{\gamma}})} \models \bar{\theta}^{I} (\dot{r}_{g,\vec{\gamma}}) \wedge \bar{\theta}^{II} (\dot{r}_{g,\vec{\gamma}}).
\end{displaymath}
Otherwise, by absoluteness, $\theta^I(\vec{\gamma}) \wedge \theta^{II}(\vec{\gamma})$ holds. Let $(H, \tau)$ witness $\theta^I(\vec{\gamma})$ and let $(K, \sigma)$ witness $\theta^{II}(\vec{\gamma})$. Let $E \in \mu_{\mathbb{L}}$ be a club such that $H \res (\rep(R) \res E) = K \res (\rep(R) \res E)$.
Let $x$ be the infinite run according to both $\tau$ and $\sigma$. Then inductively we can see that for any $\alpha < \llbracket \emptyset \rrbracket_R$, $(\alpha,x) \in B \wedge \psi(\alpha,x) < H(\rho^E(\alpha+1))$, but there is $\alpha < \llbracket \emptyset \rrbracket_R$ such that $(\alpha,x) \notin B$, which is impossible. 

In conclusion, Player I has a winning strategy in $B$ iff for any $\vec{\gamma} \in [C]^{R \uparrow}$,  $ (M_{2, \infty}^{-})^{\coll(\omega, \xi_{\vec{\gamma}})} \models \bar{\theta}^I (\dot{r}_{g,\vec{\gamma}})$.
\end{proof}

For  a real $x$, $x^{3\#}$ is the obvious relativization of $0^{3\#}$. 
Combining Lemmas~\ref{lem:3sharp_from_M2sharp} and \ref{lem:M2sharp_from_3sharp_level_by_level}, \cite[Theorem 3.1]{sharpI} and Neeman \cite{nee_opt_I,nee_opt_II}, we obtain the equivalence of $x^{3\#}$ and $M_2^{\#}(x)$.
\begin{mytheorem}
  \label{thm:M2sharp_3sharp_equivalent}
  Assume $\boldpi{3}$-determinacy. For $x \in \mathbb{R}$, $x^{3\#}$ is many-one equivalent to $M^{\#}_2(x)$, the many-one reduction being independent of $x$. 
\end{mytheorem}

By Theorem~\ref{thm:M2sharp_3sharp_equivalent} and Moschovakis third periodicity, the winner of the game in the proof of Lemma~\ref{lem:the-measure-W3R} has a winning strategy recursive in $0^{3\#}$. Hence, the set of $L$-Woodin cardinals in $M_{2,\infty}^{-}(0^{3\#})$ and their limits form a firm set of potential level-3 indiscernibles for $M_{2,\infty}^{-}$. 

\section*{Acknowledgements}
The breakthrough ideas of this paper were obtained during the AIM workshop on Descriptive inner model theory, held in Palo Alto, and the Conference on Descriptive Inner Model Theory, held in Berkeley, both in June, 2014. 
The author greatly benefited from conversations with Rachid Atmai and Steve Jackson that took place in these two conferences. 
The final phase of this paper was completed whilst the author was a visiting fellow at the Isaac Newton Institute for Mathematical Sciences in the programme `Mathematical, Foundational and Computational Aspects of the Higher Infinite' (HIF) in August and September, 2015 funded by NSF Career grant DMS-1352034 and EPSRC grant EP/K032208/1.

\bibliography{sharp}{}
\bibliographystyle{plain}

\end{document}